\documentclass{article}

\usepackage{amssymb}
\usepackage{amsmath,amsthm}
\usepackage{enumerate}
\usepackage{mathtools}
\usepackage{graphicx}
\usepackage{textcomp}
\usepackage{array}
\usepackage{color, colortbl}
\definecolor{Gray}{gray}{0.95}
\usepackage{hhline}
\usepackage{authblk}

\newcommand{\unds}{\underline{s}}
\newcommand{\Ss}{S(\unds)}

\newcommand{\undt}{\underline{t}}
\newcommand{\St}{S(\undt)}

\newcommand{\cint}{(C)\int}
\newcommand{\dm}{d\mu}
\newcommand{\muc}{\overline{\mu}}
\newcommand{\dmc}{d\muc}
\newcommand{\dlp}{d\LP}
\newcommand{\dup}{d\UP}
\newcommand{\rint}[2]{(R)\int_{#1}^{#2}}
\newcommand{\LC}{\underline{C}}
\newcommand{\UP}{\overline{P}}
\newcommand{\LP}{\underline{P}}
\newcommand{\LQ}{\underline{Q}}

\newcommand{\LE}{\underline{E}}
\newcommand{\LG}{\underline{G}}

\newcommand{\lpr}{\underline{P}}

\newcommand{\ale}{\alpha^{E}}
\newcommand{\LEC}{\LE_{c}}
\newcommand{\LEDC}{\LE_{2c}}
\newcommand{\LED}{\LE_{2}}
\newcommand{\LEP}{\LE^{+}}

\newcommand{\LGDC}{\LG_{2c}}
\newcommand{\LD}{L_{2}}
\newcommand{\LDC}{L_{2c}}

\newcommand{\rset}{\mathbb{R}}

\newcommand{\natsetp}{\mathbb{N}_{0}}

\newcommand{\lset}{\mathcal{L}}

\newcommand{\Rset}{\rset}
\newcommand{\naturals}{\mbox{I\hspace{-.1cm}N}}
\newcommand{\prt}{I\dsn P}

\newcommand{\nset}{\naturals}

\newcommand{\dset}{\mathcal{D}}
\newcommand{\aset}{\mathcal{A}}
\newcommand{\eset}{\mathcal{E}}

\newcommand{\nmset}{\mathcal{N}}
\newcommand{\xset}{\mathcal{X}}
\newcommand{\bset}{\mathcal{B}}
\newcommand{\bsetp}{\mathcal{B}^{\varnothing}}

\newcommand{\bsetz}{\mathcal{B}^{Z}}

\newcommand{\asetc}{\aset_{c}}
\newcommand{\asetpa}{\aset(\prt)}
\newcommand{\asetpav}{\asetpa^{\varnothing}}
\newcommand{\asetcond}{\asetpa|\asetpav}
\newcommand{\asetz}{\mathcal{A}^{Z}}
\newcommand{\lcond}{\xset|\bsetp}

\newcommand{\lcondz}{(\lcond)^{+}}
\newcommand{\infa}{J_{A}}
\newcommand{\infb}{J_{B}}

\newcommand{\dsn}{\!\!}

\newcommand{\comment}[1]{}

\newcommand{\nega}[1]{\neg #1}

\newcommand{\lgn}{\leq_{\textsc{GN}}}
\newcommand{\ze}{z_{E}}
{\left\lbrace\begin{array}{@{}l@{}}}%
	{\end{array}\right.}

\newtheorem{theorem}{Theorem}
\newtheorem{definition}{Definition}
\newtheorem{proposition}{Proposition}
\newtheorem{corollary}{Corollary}
\newtheorem{lemma}{Lemma}
\newtheorem{example}{Example}
\newtheorem{remark}{Remark}

\newcolumntype{C}{>{$\displaystyle} c <{$}}

\begin{document}



% title{Weak Consistency for Imprecise Conditional Previsions}
%Titolo alternativo
\title{Weakly Consistent Extensions of Lower Previsions}


\author[1]{Renato Pelessoni\thanks{renato.pelessoni@econ.units.it}}
\author[1]{Paolo Vicig\thanks{paolo.vicig@econ.units.it}}
\affil[1]{DEAMS ``B. de Finetti''\\
	University of Trieste\\
	Piazzale Europa~1\\
	I-34127 Trieste\\
	Italy}

\renewcommand\Authands{ and }

\maketitle

\begin{abstract}
Several consistency notions are available for a lower prevision $\lpr$ assessed on a set $\dset$ of gambles (bounded random variables),
ranging from the well known coherence to convexity and to the recently introduced $2$-coherence and $2$-convexity.
In all these instances,
a procedure with remarkable features,
called (coherent, convex, $2$-coherent or $2$-convex) natural extension,
is available to extend $\lpr$,
preserving its consistency properties,
to an arbitrary superset of gambles.
We analyse the $2$-coherent and $2$-convex natural extensions,
$\LED$ and $\LEDC$ respectively,
showing that they may coincide with the other extensions in certain,
special but rather common, cases
of `full' conditional lower prevision or probability assessments.
This does generally not happen if $\lpr$ is a(n unconditional) lower probability on the powerset of a given partition and is extended to the gambles defined on the same partition.
In this framework we determine alternative formulae for $\LED$ and $\LEDC$.
We also show that $\LEDC$ may be nearly vacuous in some sense,
while the Choquet integral extension is $2$-coherent if $\lpr$ is,
and bounds from above the $2$-coherent natural extension.
Relationships between the finiteness of the various natural extensions and conditions of avoiding sure loss or weaker are also pointed out.

\end{abstract}

\section*{Acknowledgement}
*NOTICE: This is the authors' version of a work that was accepted for publication in Fuzzy Sets and Systems. Changes resulting from the publishing process, such as peer review, editing, corrections, structural formatting, and other quality control mechanisms may not be reflected in this document. Changes may have been made to this work since it was submitted for publication. A definitive version was subsequently published in Fuzzy Sets and Systems, 
vol. 328, December~2017, pages 83–-106
10.1016/j.fss.2017.07.021 $\copyright$ Copyright Elsevier

%http://www.sciencedirect.com/science/article/pii/S0888613X16300792.
https://doi.org/10.1016/j.fss.2017.07.021

\vspace{0.3cm}
$\copyright$ 2017. This manuscript version is made available under the CC-BY-NC-ND 4.0 license http://creativecommons.org/licenses/by-nc-nd/4.0/

\begin{center}
	\includegraphics[width=2cm]{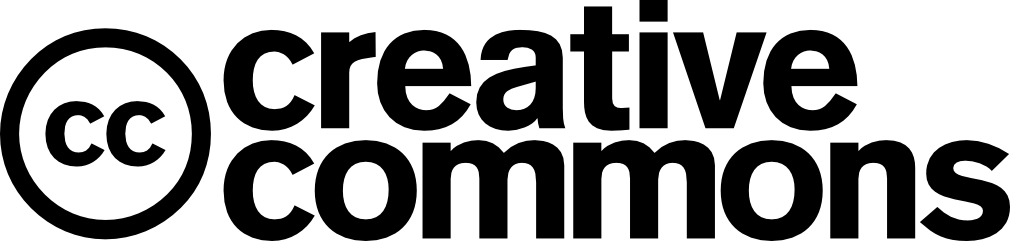}
	\includegraphics[width=2cm]{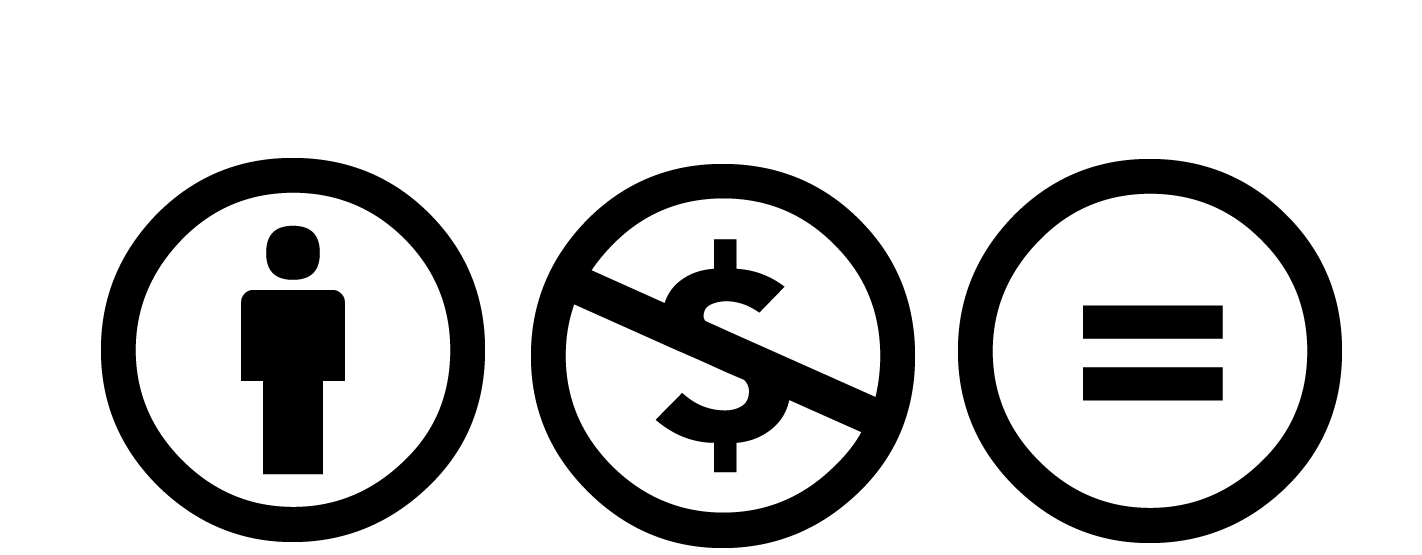}
\end{center}

\section{Introduction}
\label{sec:introduction}
Imprecise Probability Theory,
as developed in \cite{wal91,wil07}
(see also \cite{tro14,wal81})
exploits the notion of \emph{coherence}
as its main consistency concept for an uncertainty measure.
The measure is a lower (or upper) probability or, more generally, prevision if applied to an arbitrary set $\dset$ of bounded random variables or
\emph{gambles}.
The set $\dset$ may be also made of conditional gambles such as $X|B$,
with $X$ (unconditional) gamble and $B$ non-impossible event.
Coherent lower and upper previsions encompass a number of uncertainty models as special cases,
including coherent (precise) previsions in the sense of de Finetti \cite{def74},
classical probabilities, $2$-monotone probabilities,
coherent risk measures, belief functions and others.

Weaker consistency concepts have been investigated in order to accommodate further uncertainty models into a unit framework.
The framework we refer to is the \emph{betting scheme}
(described in Section \ref{sec:preliminaries}),
which goes back to de Finetti's approach to subjective probability \cite{def74} and
underlies, with variants, all these concepts.
Thus, convex lower previsions were studied in \cite{pel05},
showing that they include convex risk measures and other models.

In a recent paper \cite{pel15} we introduced two still weaker consistency concepts for
conditional lower previsions,
\emph{$2$-convexity} and \emph{$2$-coherence},
studying their basic properties in greater detail in \cite{pel16}.
Formally,
$2$-coherent and $2$-convex conditional lower previsions are a broad generalisation
of the $2$-coherent (unconditional) lower previsions in \cite[Appendix B]{wal91}.
Normalised capacities and niveloids are instances of $2$-convex lower previsions
\cite[Section 5]{pel16}.

A fundamental question is to detect
which properties from stronger consistency concepts are somehow retained by either $2$-convexity or $2$-coherence.
As shown in \cite{pel15, pel17},
a very relevant feature of theirs is that they are endowed with,
respectively,
a $2$-convex and a $2$-coherent \emph{natural extension}.
The properties of these extensions, exemplified in Proposition \ref{pro:properties_2_ne}, are formally
perfectly analogous to those of the natural extension for coherent lower previsions
(following Williams' coherence in the conditional framework \cite{wil07})
or the convex natural extension for convex conditional previsions \cite{pel05}.
In particular,
when finite,
they allow extending a lower prevision $\lpr$ from its domain $\dset$ to any larger $\dset^\prime\supset\dset$,
preserving its consistency properties.
For instance,
the $2$-coherent natural extension is $2$-coherent on $\dset^\prime$ if $\lpr$ is so on $\dset$,
and points out which commitments on our evaluation on $\dset^\prime$ are implied by the evaluation $\lpr$ on $\dset$.
Thus all these natural extensions solve a basic inferential problem.

The main purpose of this paper is to investigate further the features of the $2$-coherent and $2$-convex natural extension and to compare them with the natural extension for coherent or convex previsions.

A first consideration is immediate: since
$2$-coherence is weaker than coherence, inferences produced by the $2$-coherent natural extension will be generally vaguer than those
guaranteed by the coherent natural extension for the same $\lpr$,
and similarly with $2$-convexity versus convexity or coherence.
%Actually,
%often $2$-coherent or $2$-convex natural extensions will be even too vague.
This points out a drawback of these weak consistency notions and is one reason why,
in our view,
they should not always be regarded as realistic candidates for replacing coherence or convexity.

Yet, the $2$-convex (termed $\LEDC$) and $2$-coherent ($\LED$) natural extensions may be helpful in determining the coherent ($\LE$) or convex ($\LEC$) natural extensions.
In fact,
there are significant instances where some or all of the four extensions  coincide.

After concisely presenting the necessary preliminary notions in Section \ref{sec:preliminaries},
we investigate one such case in Section \ref{sec:extensions_full}.
Here $\lpr$ is a lower conditional probability assessed in the `full' environment $\asetc$ of Definition \ref{def:special_structures}.
It is shown in Proposition \ref{pro:ext_prob_many} that the lower Goodman-Nguyen extension,
a standard and easy-to-apply procedure investigated in \cite{pel14},
corresponds to the four natural extensions onto any set of conditional events.
Hence they coincide,
if $\lpr$ is coherent on $\asetc$ (Proposition \ref{pro:sandwich_thm}).
The reader will notice,
entering the detail of Section \ref{sec:extensions_full},
that proofs for its results mostly refer to \cite{pel14}.
The reason for this is that \cite{pel14} was written before our introducing $2$-coherent and $2$-convex previsions in \cite{pel15,pel16}.
Indeed,
one motivation for their introduction was precisely the fact that most proofs in \cite{pel14} required weaker consistency assumptions than we initially expected.

In Section \ref{sec:extension_cond_low_prev},
a lower prevision $\lpr$ is initially defined on a structured set
$\lcond$ (cf. Definition \ref{def:def_dlin}) of conditional gambles,
representing a generalisation of a vector space to a conditional environment.
Hence we are considering a special,
but rather common,
situation.
In Proposition~\ref{pro:alt_nat_ext} we give an alternative expression for the coherent natural extension,
which is later needed and generalises a result in \cite{wal91} (cf. Corollary \ref{co:Walley_unconditional}).
After showing how to ensure finiteness for the relevant natural extensions,
Theorems \ref{thm:ext_coer}, \ref{thm:ext_conv}, \ref{thm:final_sandwich} and
Corollary~\ref{cor:equality_e_e2} present
instances where more different extensions coincide.
These results are discussed after Theorem \ref{thm:final_sandwich}.
The contents of this section deepen those of \cite{pel17},
where proofs were only partly supplied.

In Section \ref{sec:extension_lprob_to_lprev},
$\lpr$ is a lower probability assessment on the set $\asetpa$ of events logically dependent on an arbitrary partition $\prt$
(i.e. the powerset of $\prt$ in set-theoretical language),
and we consider its extensions to the gambles defined on $\prt$.
We first obtain two different formulae for the $2$-coherent natural extension $\LED$
(Section \ref{subsec:2_coherent_ne}) and
two for the $2$-convex natural extension $\LEDC$ (Section \ref{subsec:2_convex_ne}),
then discuss them  and a further formula for $\LEDC$ in Section \ref{subsec:more_on_ne}.
It turns out that the formulae operationally simplify the computation of $\LED$, $\LEDC$ when $\prt$ is finite,
while in general,
for any $\prt$ and a gamble $Z$,
$\LED$ and $\LEDC$ both depend on the assessment of the lower decumulative distribution function of $Z$,
$\lpr(Z\geq z)$, $z\in\rset$.
In this framework,
it is also possible to evaluate how vacuous $\LED$ may be:
for any gamble $Z$, it holds that $\LEDC(Z)\leq\inf Z+1$.
This lets us guess that equality results for different extensions comparable to those of the previous sections are not to be expected now.
And in fact,
in Section \ref{subsec:Choquet_extension} we discuss a different extension for a $2$-coherent $\lpr$,
the Choquet integral extension $\LC$,
showing that it is $2$-coherent (Proposition \ref{pro:Choquet_2c}) and that $\LC\geq\LED$,
the inequality being generally strict.
Since it is known that $\LC\leq\LE$,
with equality iff $\lpr$ is coherent and $2$-monotone \cite{wal81},
it will generally be the case that $\LE>\LED$.
In the discussion following Example \ref{exa:not_coincide} we note that the results in Section \ref{subsec:Choquet_extension} contribute to fixing the role of the Choquet integral extension: next to being a lower bound for $\LE$,
it is an upper bound for $\LED$.

While the assumptions in each of Sections \ref{sec:extensions_full}, \ref{sec:extension_cond_low_prev}, \ref{sec:extension_lprob_to_lprev}
are sufficient for the relevant extensions to be finite,
in Section \ref{sec:more_order_ne} we discuss characterisation for their finiteness (introduced in \cite{pel03bis, pel16, wal91}),
in terms of avoiding sure loss and weaker conditions,
in an unconditional setting.
We provide examples showing that even when some extensions are infinite
their partial ordering from Lemma \ref{lem:order_ne} cannot be improved.
Section \ref{sec:conclusions} summarizes the paper.

\section{Preliminaries}
\label{sec:preliminaries}
Let $\dset$ be an arbitrary set of conditional gambles,
that is,
the generic element of $\dset$ is $X|B$,
with $X$ gamble (a bounded random variable),
$B$ non-impossible event.
A \emph{conditional lower prevision} $\lpr:\dset\rightarrow\rset$ is a real map which,
behaviourally,
determines the supremum buying price $\lpr(X|B)$ of any $X|B\in\dset$.
This means that an agent should be willing to buy, or to bet in favour of, $X|B$,
for any price lower than $\lpr(X|B)$.
The agent's \emph{elementary gain} from the transaction/bet on
$X|B$ for $\lpr(X|B)$ is $I_B (X-\lpr(X|B))$.
Here $I_B$ is the indicator of event $B$.
Its role is that of ensuring that the purchased bet is called off and the money returned to the agent iff $B$ does not occur.
In the sequel,
we shall use the symbol $B$ for both event $B$ and its indicator $I_B$.

\subsection{Consistent lower previsions}
\label{subsec:consistent_lp}
A generic consistency requirement for $\lpr$ 
asks that no finite linear combination of elementary gains on elements of $\dset$,
with prices given by $\lpr$,
should produce a loss (bounded away from $0$) for the agent.
We obtain different known concepts by imposing constraints either on the number of terms in the linear combination or on their coefficients $s_i$:
\begin{definition}
	\label{def:consistency}
	Let $\LP:\dset\rightarrow\Rset$ be a given conditional lower prevision.
	\begin{itemize}
		\item[a)]
		$\LP$ is a \emph{coherent} conditional lower prevision on $\dset$ iff,
		for all $m\in\natsetp$,
		$\forall X_0|B_0,\ldots,X_m|B_m\in\dset$, $\forall s_0,\ldots,s_m\geq 0$,
		defining
		$\Ss=\bigvee\{B_{i}:s_{i}\neq 0, i=0,\ldots,m\}$ and
		$$\LG=\sum_{i=1}^{m}s_{i}B_{i}(X_{i}-\LP(X_{i}|B_{i}))-s_{0}B_{0}(X_{0}-\LP(X_{0}|B_{0})),$$
		it holds,
		whenever $\Ss\neq\varnothing$,
		that $\sup\{\LG|\Ss\}\geq 0$.
		\item[b)]
		$\lpr$ is \emph{$2$-coherent} on $\dset$ iff a) holds
		with $m=1$ (hence there are \emph{two} terms in $\LG$).
		\item[c)]
		$\lpr$ is \emph{convex} on $\dset$ iff a) holds with the
		additional convexity constraint $\sum_{i=1}^{m}s_i=s_0=1$.
		\footnote{
		The term \emph{convex} originates in our framework from the convexity constraint
		$\sum_{i=1}^{m}s_i=1$,
		differentiating convex lower previsions from coherent ones.
		Thus, convex lower previsions should not be confused with convex (or $2$-monotone) fuzzy measures, which under mild normalisation constraints are instead a special subset of coherent lower previsions.
		In this latter meaning, the term convex is often used as a synonym of supermodular
		\cite[p. 16]{den94},
		although this may be in general improper \cite[Definition 2.18 (ii)]{gra16}.		
		}
		\item[d)]
		$\lpr$ is \emph{$2$-convex} on $\dset$ iff c) holds with $m=1$,
		i.e., iff, $\forall X_0|B_0, X_1|B_1\in\dset$,
		we have that,
		defining $\LGDC=B_{1}(X_{1}-\LP(X_{1}|B_{1}))-B_{0}(X_{0}-\LP(X_{0}|B_{0}))$, $\sup(\LGDC|B_{0}\vee B_{1})\geq 0$.
		\item[e)]
		$\lpr$ is \emph{centered}, convex or $2$-convex, on $\dset$ iff it is convex or $2$-convex, respectively,
		and $\forall X|B\in\dset$, it is $0|B\in\dset$ and $\lpr(0|B)=0$.
	\end{itemize}
\end{definition}
Condition a),
which is Williams' coherence \cite[Definition 1]{wil07} in the stru--cture-free version of \cite[Definition 4]{pel09},
is clearly the strongest one.
Convexity is a relaxation of coherence, studied in \cite[Definition 12]{pel05}.
Given $\lpr$ on $\dset$,
the following relationships hold:
\begin{eqnarray}
	\label{eq:consistency_implications}
	\begin{array}{ll}
		\lpr \mbox{ coherent} \Rightarrow \lpr \mbox{ $2$-coherent} \Rightarrow \lpr \mbox{ $2$-convex}\\
		\lpr \mbox{ coherent} \Rightarrow \lpr \mbox{ convex} \Rightarrow \lpr \mbox{ $2$-convex}
	\end{array}
\end{eqnarray}
The consistency concepts recalled so far can be characterised by means of axioms on the special sets $\lcond$ defined next:
\begin{definition}
	\label{def:def_dlin}
	Let $\xset$ be a linear space of gambles and
	$\bset\subset\xset$ a set of (indicators of) events, such that
	$\Omega\in\bset$ and
	$\xset$ is stable by restriction,
	i.e. $BX\in\xset, \forall B\in\bset, \forall X\in\xset$.
	Setting $\bsetp=\bset\setminus\{\varnothing\}$,
	define
	$\lcond=\{X|B: X\in\xset, B\in\bsetp\}$.
\end{definition}
The sets $\lcond$ generalise linear spaces (obtained for $\bset=\{\varnothing,\Omega\}$)
containing real constants
(note for this that we may equivalently write $1=I_\Omega\in\bset$
instead of $\Omega\in\bset$ in Definition \ref{def:def_dlin}).
\begin{theorem}[Characterisation Theorems]
	\label{thm:characterisation}
	Let $\lpr:\lcond\rightarrow\rset$ be a conditional lower prevision.
	\begin{itemize}
		\item[a)]
		$\lpr$ is coherent on $\lcond$ if and only if \cite[Theorem 2]{pel09}
		\cite[p. 370]{wil07}
		\begin{itemize}
			\item[(A1)] $\quad\ \LP(X|B)-\LP(Y|B)\leq\sup\{X-Y|B\}$,
			$\forall X|B, Y|B\in\lcond$.
			\item[(A2)] $\quad\ \LP(\lambda X|B)=\lambda\LP(X|B),
			\forall X|B\in\lcond, \forall\lambda\geq 0$.
			\item[(A3)] $\quad\ \LP(X+Y|B)\geq\LP(X|B)+\LP(Y|B)$,
			$\forall X|B$, $Y|B\in\lcond$.
			\item[(A4)] $\quad\ \LP(A(X-\LP(X|A\wedge B))|B)=0,
			\forall X\in\xset,
			\forall A, B\in\bset:
			A\wedge B\neq\varnothing$.
		\end{itemize}
		\item[b)]
		$\lpr$ is $2$-coherent on $\lcond$ if and only if (A1), (A2), (A4) and the following axiom hold \cite[Proposition 9]{pel16}:
		\begin{itemize}
			\item[(A5)] $\quad\ \lpr(\lambda X|B)\leq \lambda\lpr(X|B)$, $\forall\lambda<0$.
		\end{itemize}
		\item[c)]
		$\lpr$ is convex on $\lcond$ if and only if (A1), (A4) and the following axiom hold  \cite[Theorem 8]{pel05}
		\begin{itemize}
			\item[(A6)] $\quad\ \LP(\lambda X+(1-\lambda)Y|B)\geq\lambda\LP(X|B)+(1-\lambda)\LP(Y|B),
				\forall X|B,Y|B\in\lcond, \forall \lambda\in (0,1)$.
		\end{itemize}
		\item[d)]
		$\lpr$ is $2$-convex on $\lcond$ if and only if (A1) and (A4) hold \cite[Proposition 2]{pel16}.
	\end{itemize}
\end{theorem}
Every single axiom in Theorem \ref{thm:characterisation} is a \emph{necessary}
consistency requirement when $\lpr$ is defined on a generic set $\dset$
(including the relevant gambles).
This fact will often be useful in later proofs.

There is a number of further necessary conditions for the consistency of $\lpr$. Among these the following will be useful later and are implied by axiom (A1):
\begin{align}
	\label{eq:intern}
	&\inf X|B\leq\lpr(X|B)\leq\sup X|B, \forall X|B\in\dset\\
	\label{eq:trans}
	&\lpr(X+c\ |B)=\lpr(X|B)+c, \forall X|B\in\dset, \forall c\in\rset\\
	\label{eq:monot}
    &X|B\leq Y|B \mbox{ implies } \lpr(X|B)\leq\lpr(Y|B), \forall X|B,Y|B\in\dset.
\end{align}

A lower probability is formally the special instance of lower prevision that is defined on a set $\dset$ of gambles that are all indicators of (possibly conditional) events.
While $\dset$ may again be arbitrary,
in this paper we shall focus on the special domains for lower probabilities
defined next.
\begin{definition}
	\label{def:special_structures}
	Given an arbitrary partition $\prt$,
	$\asetpa$ is the set of all events logically dependent on $\prt$
	(the powerset of $\prt$, in set-theoretical language),
	$\asetpav=\asetpa\setminus\{\varnothing\}$,
	and $\asetc=\asetpa|\asetpav=\{A|B: A\in\asetpa, B\in\asetpav\}$.
\end{definition}
$\asetpa$ (and consequently, to a large extent, $\asetc$) is a `full' environment,
in the sense that it is closed under the logical operators
$\nega{ }$, $\vee$, $\wedge$.

\subsection{Natural extensions}
\label{subsec:natural_extensions}
Next we recall the definitions of the different natural extensions that appear in this paper.
%We ??? that `the' natural extension,
%without further qualifications, is
The term `natural extension',
without further qualifications,
will denote
the coherent natural extension in Definition \ref{def_nat_ext}, a).
\begin{definition}[Various natural extensions]
	\label{def_nat_ext}
	Let $\lpr:\dset\rightarrow\rset$ be a conditional lower prevision,
	and $Z|A$ a conditional gamble.
	\begin{itemize}
		\item[a)]
		Define
		\begin{align*}
			L(Z|A)=&\{\alpha\in\rset:
			\sup\{\sum_{i=1}^{m}s_i B_i (X_i-\lpr(X_i|B_i))-A(Z-\alpha)|A\vee\Ss\}\\
			&<0, \mbox{ for some } m\geq 0, X_i|B_i\in\dset, s_i\geq 0, i=1,\ldots,m\},
		\end{align*}
		where $\Ss=\vee_{i=1}^{m}\{B_i:s_i\neq 0\}$.
		Then, the \emph{(coherent)} natural extension of $\lpr$ on $Z|A$ is
		$\LE(Z|A)=\sup L(Z|A)$.
		\item[b)]
		Define $L_2 (Z|A)$ putting $m=1$ in $L(Z|A)$,
		i.e.
		\begin{align*}
			\LD(Z|A)=&\{\alpha\in\rset:\sup\{s_1 B_1 (X_1-\lpr(X_1|B_1))- A(Z-\alpha)|A\vee\Ss\}\\
			&<0, \mbox{ for some } X_1|B_1\in\dset, s_1\geq 0\}.
		\end{align*}		 
		The \emph{$2$-coherent} natural extension of $\lpr$ on $Z|A$ is
		$\LED (Z|A)=\sup L_2 (Z|A)$. 
		\item[c)]
		Define $L_c (Z|A)$ from $L(Z|A)$,
		by adding the constraint $\sum_{i=1}^{m}s_i=1$, i.e.
		%in the	`for some' part.
		\begin{align*}
		L_c(Z|A)=&\{\alpha\in\rset:
		\sup\{\sum_{i=1}^{m}s_i B_i (X_i-\lpr(X_i|B_i))-A(Z-\alpha)|A\vee\Ss\}\\
		&<0, \mbox{ for some } m\geq 0, X_i|B_i\in\dset, s_i\geq 0, i=1,\ldots,m\\
		&\mbox{ with } \sum_{i=1}^{m}s_i=1\}.
		\end{align*}
		The \emph{convex} natural extension of $\lpr$ on $Z|A$ is
		$\LEC (Z|A)=\sup L_c (Z|A)$.
		\item[d)]
		Define $L_{2c}(Z|A)$ putting $m=1$ in $L_{c}(Z|A)$, 
		i.e.
		\begin{align*}
			L_{2c}(Z|A)=&\{\alpha\in\rset:\sup\{B(X-\lpr(X|B))-A(Z-\alpha)|A\vee B\}<0,\\
			&\mbox{for some} X|B\in\dset\}.
		\end{align*}
		The \emph{$2$-convex} natural extension $\LE_{2c}$ of $\lpr$ on $Z|A$ is $\LEDC=\sup L_{2c}(Z|A)$.
	\end{itemize}
\end{definition}

The properties of these four natural extensions are analogous
\cite{pel05, pel09, pel16}.
Here we state them for the $2$-convex natural extension.
For the properties of $\LE$, $\LED$, $\LEC$, replace $\LEDC$ and `$2$-convex'
with, respectively,
$\LE$ and `coherent', $\LED$ and `$2$-coherent', $\LEC$ and `convex'.

\begin{proposition}
	\label{pro:properties_2_ne}
	Let $\lpr:\dset\rightarrow\rset$ a conditional lower prevision,
	with $\dset\subset\lcond$.
	If $\LEDC$ is finite on $\lcond$, then
	\begin{itemize}
		\item[a)] $\LEDC(Z|A)\geq\lpr(Z|A)$, $\forall Z|A\in\dset$.
		\item[b)] $\LEDC$ is $2$-convex on $\lcond$.
		\item[c)] If $\lpr^{*}$ is $2$-convex on $\lcond$ and $\lpr^{*}(Z|A)\geq\lpr(Z|A)$, $\forall Z|A\in\dset$,
		then $\lpr^{*}(Z|A)\geq\LEDC(Z|A)$, $\forall Z|A\in\lcond$.
		\item[d)] $\lpr$ is $2$-convex on $\dset$ if and only if $\LEDC=\lpr$ on $\dset$.
		\item[e)] If $\lpr$ is $2$-convex on $\dset$, $\LEDC$ is its smallest $2$-convex extension on $\lcond$.
	\end{itemize}
\end{proposition}
By Proposition \ref{pro:properties_2_ne},
$\LEDC$ dominates $\lpr$ on $\dset$ (by a)),
characterises $2$-convexity (by d)) and
is the least-committal $2$-convex extension of $\lpr$ on $\dset^*$ (by c), e)).
Similarly for $\LE$, $\LEC$, $\LED$.

Given a lower prevision $\lpr$ on $\dset$,
its natural extensions $\LE$, $\LED$, $\LEC$, $\LEDC$ will generally be different, and ordered as follows.
\begin{lemma}
	\label{lem:order_ne}
	Given $\lpr:\dset\rightarrow\rset$, it holds that
	\begin{eqnarray}
		\label{eq:order_ne}
		\begin{array}{ll}
			\LE\geq\LED\geq\LEDC\\
			\LE\geq\LEC\geq\LEDC
		\end{array}
	\end{eqnarray}
\end{lemma}
\begin{proof}
	It is easy to realise that \eqref{eq:order_ne} holds recalling
	Definition \ref{def_nat_ext},
	since it implies, $\forall Z|A$,
	$\LDC(Z|A)\subseteq\LD(Z|A)\subseteq L(Z|A)$ and
	$\LDC(Z|A)\subseteq L_c(Z|A)\subseteq L(Z|A)$.
\end{proof}
It may also be the case that some among $\LE$, $\LED$, $\LEC$, $\LEDC$ are infinite.
But even when being finite,
some or all of them may differ also considerably,
while other ones may possibly coincide.
For an illustration, see the next simple example.
(Comparisons involving infinite extensions are postponed to Section \ref{sec:more_order_ne}.)
\begin{example}
	\label{exa:different_ne}
	Let $\dset=\{X\}$,
	where $X$ may only take the values $0$ and $1$.
	Assign $\lpr(X)\in (0,1)$,
	which is clearly coherent, hence $2$-coherent, convex and $2$-convex, on $\dset$.
	Its natural extension $\LE$ on $\{2X\}$ is $\LE(2X)=2\lpr(X)$
	by (A2),
	because $\LE$ is coherent on $\{X,2X\}$ and coincides with $\lpr$ on $X$.
	By the same reasoning, also $\LED(2X)=2\lpr(X)$.
	
	As for $\LEC(2X)$ and $\LEDC(2X)$,
	recalling Definition \ref{def_nat_ext} b) and d), we get
	\begin{align*}
	L_c (2X)=\LDC(2X)=&\{\alpha:\sup\{(X-\lpr(X)-(2X-\alpha))\}<0\}=\\
	&\{\alpha:\alpha<\inf X+\lpr(X)\}=(-\infty,\lpr(X)).
	\end{align*}
	Hence, $\LEC(2X)=\LEDC(2X)=\lpr(X)$.
	%applying (A1) with $Y=2X$,
	%$\lpr(2X)\geq\lpr(X)-\sup(X-2X)=\lpr(X)+\inf X=\lpr(X)$,
	%hence $\LEDC(2X)\geq\lpr(X)$.
	%Then,
	%it can be checked directly using Definition \ref{def:consistency}, d) that
	%$\LEDC(2X)=\lpr(X)$ is $2$-convex (there are only two gains $\LGDC$ to %inspect).
	%Further, $\LEC(2X)=\lpr(X)$:
	%apply the same argumentation of $\LEDC$,
	%extending the number of gains to inspect to four.
	%In fact, one of the additional gains is
	%$s_1(X-\lpr(X))+s_2(2X-\lpr(X))-(X-\lpr(X))=
	%(s_1+2 s_2 - 1)X+(1-s_1 - s_2)\lpr(X)=s_2 X\geq 0$,
	%the other is
	%$s_1 (X-\lpr(X))+s_2 (2X-\lpr(X))-(2X-\lpr(X))=(s_1+2 s_2-2)X=-s_1 X\leq 0$,
	%which is $0$ when $X$ is equal to $0$.
	
	Summing up, in this example we have $\LE(2X)=\LED(2X)>\LEC(2X)=\LEDC(2X)$.
\end{example}
The following sufficient conditions will guarantee the finiteness of
$\LE$, $\LED$, $\LEC$, $\LEDC$ in the next three sections.
\begin{proposition}[\cite{pel05, pel16}]
	\label{pro:finiteness_ne}
	Given $\lpr:\dset\rightarrow\rset$ and $Z|A$:
	%\begin{itemize}
	%	\item[a)]
	%	$\LE(Z|A)$ ($\LED(Z|A)$) is finite, $\forall Z|A$,
	%	if $\lpr$ is coherent ($2$-coherent).
	%	\item[b)]
	%	$\LEC(Z|A)$, $\LEDC(Z|A)$ are finite, $\forall Z|A$,
	%	if $0|A\in\dset$ and $\lpr(0|A)=0$.		
	%\end{itemize}
	\begin{itemize}
		\item[a)]
		If $\lpr$ is coherent ($2$-coherent),
		$\LE(Z|A)$ ($\LED(Z|A)$) is finite.
		\item[b)]
		%If, $\forall Z|A$, $0|A\in\dset$ and $\lpr(0|A)=0$,
		%then $\LEC(Z|A)$, $\LEDC(Z|A)$ are finite.	
		If $\lpr$ is centered convex (centered $2$-convex) and $0|A\in\dset$,
		then $\LEC(Z|A)$ ($\LEDC(Z|A)$) is finite.
	\end{itemize}
\end{proposition}

\subsection{The Choquet integral}
\label{subsec:Choquet_integral}
In Section \ref{subsec:Choquet_extension} we shall utilise the Choquet integral. This integral is studied in a general context e.g. in \cite{den94}.
The shorter presentation in \cite[Appendix C]{tro14} is very clear and by far sufficient for the developments in this paper.
Actually,
we shall only need the notions briefly recalled here,
adapted from \cite{tro14}.

Let $\mu:\asetpa\rightarrow\rset$ be an uncertainty measure which is \emph{monotone},
meaning that $A\Rightarrow B$ implies $\mu(A)\leq\mu(B)$,
and \emph{normalised}, i.e. $\mu(\varnothing)=0$, $\mu(\Omega)=1$,
that is $\mu$ is a (normalised) capacity.
The \emph{conjugate} $\muc$ of $\mu$, defined as $\muc(A)=\mu(\Omega)-\mu(\nega{A})$,
$\forall A\in\asetpa$,
is monotone and normalised too, as can be easily checked.

Defining $\lset$ as the set of all gambles on $\prt$,
for each $X\in\lset$ we can evaluate
\begin{eqnarray}
\label{eq:Choquet_def}
(C) \int Xd\mu=(R)\int_{0}^{+\infty}\mu(X^+\geq x)dx-(R)\int_{0}^{+\infty}\muc(X^-\geq x)dx
\end{eqnarray}
where $X^+=\max(X,0)$, $X^-=-\min(X,0)$,
and the right-hand side integrals are generalised Riemann integrals.
Since $X$ is bounded, their difference $(C)\int Xd\mu$ can be shown to be well-defined \cite{tro14} and is called the \emph{Choquet integral} of $X$ with respect to $\mu$.
\begin{proposition}[Basic properties of the Choquet integral (\cite{tro14}, Appendix C, Section C.3)]
\label{pro:Choquet_int_prop}
Let $X\in\lset$. Then
\begin{itemize}
\item[a)]
${\displaystyle \cint I_A \dm=\mu(A), \forall A\in\asetpa}$
\item[b)]
${\displaystyle\cint\lambda X\dm=\lambda\cdot\cint X\dm, \forall \lambda\geq 0}$ (non-negative homogeneity)
\item[c)]
${\displaystyle\cint -X\dm=-\cint X\dmc}$
\item[d)]
${\displaystyle\cint (X+k)\dm=\cint X\dm + k, \forall k\in\rset}$ (constant additivity)
\item[e)]
${\displaystyle\cint X\dm\leq\cint Y\dm, \forall X, Y: X\leq Y}$ (monotonicity)
\item[f)]
${\displaystyle\cint X\dm=\inf X+\rint{\inf X}{\sup X} \mu(X\geq x)dx}$
%\begin{eqnarray*}
%\label{eq:Choquet_alternative}
%\cint X\dm=\inf X+\rint{\inf X}{\sup X} \mu(X\geq x)dx
%\end{eqnarray*}
\item[g)]
If $X$ is simple, i.e. has finitely many distinct values $x_1<x_2<\ldots<x_n$, then
\begin{eqnarray*}
\label{eq:Choquet_simple}
\cint X\dm=x_1+\sum_{i=2}^{n}(x_i-x_{i-1})\mu(X\geq x_i)
\end{eqnarray*}
\end{itemize}
\end{proposition}

\section{Goodman-Nguyen extensions of `full' conditional lower probabilities}
\label{sec:extensions_full}
The starting point of the extension problems we consider in this section is a lower probability assignment on $\asetc=\asetcond$, where $\prt$ is a given arbitrary partition (cf. Definition \ref{def:special_structures}).
We wish to extend $\lpr$ on $\eset\supseteq\asetc$, arbitrary set of conditional events,
in a least-committal way and preserving the consistency properties that we may suppose $\lpr$ already has on $\asetc$.

In the case that $\lpr$ is coherent or centered convex on $\asetc$,
this problem has been solved in \cite[Section 3.1]{pel14} by means of the \emph{Goodman-Nguyen (GN) relation} and the related concepts
of \emph{inner} and \emph{outer conditional event},
and of \emph{lower} and \emph{upper GN-extension}.
Let us briefly recall these notions.

The \emph{Goodman-Nguyen relation} $\lgn$ is a generalisation to conditional events
of the implication relation (or inclusion, in set-theoretic language) \cite{goo88, pel14}:
\begin{eqnarray}
\label{eq:GN_relation}
A|B\lgn C|D \mbox{ iff } A\wedge B\Rightarrow C\wedge D \mbox{ and } \nega{C}\wedge D\Rightarrow \nega{A}\wedge B.
\end{eqnarray}
Given the partition $\prt$,
take an arbitrary conditional event $C|D$, not necessarily belonging to $\asetc$
(the interesting case is $C|D\notin\asetc$, as we plan to extend $\lpr$ to $C|D$).
We suppose hereafter that $C|D$ is \emph{non-trivial}
($C\wedge D\neq\varnothing$, $C\wedge D\neq D$).

Define
$m(C|D)=\{A|B\in\asetc: A|B\lgn C|D\}$,
and $M(C|D)=\{A|B\in\asetc: C|D\lgn A|B\}$.

Recalling that,
given an event $E$,
its \emph{inner event} is $E_*=\bigvee\limits_{e \in\prt:\ e\Rightarrow E} e$
and its \emph{outer event} is
$E^*=\bigvee\limits_{e\in\prt:\ e\wedge E\neq\varnothing} e$,
we have that \cite{pel14}:
\begin{proposition}
	\label{pro:set_asterisks}
	Given $C|D$, the sets $m(C|D)$, $M(C|D)$ are non-empty and have,
	respectively, a maximum $(C|D)_*$ and a minimum $(C|D)^*$
	conditional event w.r.t. $\lgn$,
	\begin{equation}
	\label{eq:max_cond_event}
	\begin{array}{lll}
	(C|D)_*=(C\wedge D)_*|[(C\wedge D)_*\vee(\nega{C}\wedge D)^*],\\
	(C|D)^*=(C\wedge D)^*|[(C\wedge D)^*\vee(\nega{C}\wedge D)_*].
	\end{array}
	\end{equation}
\end{proposition}
$(C|D)_*$ and $(C|D)^*$ are the \emph{inner}, respectively \emph{outer conditional event} of $C|D$ and
\eqref{eq:max_cond_event} shows that they are obtained as functions of
unconditional inner and outer events.
Obviously,
$(C|D)_*=(C|D)^*=C|D$ iff $C|D\in\asetc$.

In general,
we term \emph{lower GN-extension} (\emph{upper GN-extension}) of $\lpr$ on $\eset$ the extension defined by
$\lpr_{L}(C|D)=\lpr((C|D)_*)$
(respectively $\lpr_{U}(C|D)=\lpr((C|D)^*)$), $\forall C|D\in\eset$.

As already hinted in the Introduction,
the results in \cite{pel14} about GN--extensions apply,
with minor variations in their proofs, to $2$-coherent and centered $2$-convex
lower (and upper) previsions.
They are fundamental in solving the extension problems of this section,
as we are now going to show.

Suppose first that $\eset=\asetc\cup\{C|D\}$,
that is we add just one event to $\asetc$.
The following proposition generalises \cite[Proposition 4]{pel14},
and both results may be viewed as imprecise probability versions of de Finetti's
\emph{Fundamental Theorem of Probability} \cite{def37}.
\begin{proposition}
	\label{pro:ext_prob}
	Let $\lpr$ be $2$-coherent (alternatively, centered $2$-convex) on $\asetc$.
	Any extension $\lpr^\prime$ of $\lpr$ on $\eset=\asetc\cup\{C|D\}$ is $2$-coherent (alternatively, centered $2$-convex)
	if and only if $\lpr^\prime (C|D)\in [\lpr((C|D)_*), \lpr((C|D)^*)]$.
\end{proposition}
\begin{proof}
Same as the proof of Proposition 4 in \cite{pel14},
noting that the assumption there,
that $\lpr$ is either coherent or centered convex,
can be relaxed to the current hypotheses on $\lpr$,
and that the same applies to Lemma 1 in \cite{pel14},
recalled in that proof.
\end{proof}
In Proposition \ref{pro:ext_prob},
the lower GN-extension $\lpr((C|D)_*)$ is the $2$-coherent or $2$-convex \emph{natural extension} of $\lpr$,
being its least-committal extension on $C|D$.
This means that any other $2$-coherent or, respectively, $2$-convex
prevision $\LQ\geq\lpr$ on $\asetc$ is such that $\LQ (C|D)\geq\lpr((C|D)_*)$,
cf. also Proposition \ref{pro:properties_2_ne}~c).
The upper GN-extension $\lpr((C|D)^*)$ has the opposite interpretation
of ($2$-coherent or $2$-convex) \emph{upper extension},
corresponding to the notion studied in \cite[Section 2.8]{wei01} for coherent lower probabilities.
For a general $\eset$,
the lower GN-extension still coincides with the $2$-coherent or $2$-convex natural extension:
\begin{proposition}
	\label{pro:ext_prob_many}
	Let $\lpr$ be a coherent, or alternatively centered convex, $2$-coherent or
	centered $2$-convex lower prevision defined on $\asetc$,
	and $\eset\supseteq\asetc$ an arbitrary set of conditional events.
	Then, the lower GN-extension of $\lpr$ on $\eset$
	($\lpr_{L} (C|D)=\lpr((C|D)_*)$, $\forall C|D\in\eset$),
	is the natural extension or, respectively, the convex,
	$2$-coherent or $2$-convex natural extension of $\lpr$ on  $\eset$.
\end{proposition}
\begin{proof}
For a coherent or centered convex $\lpr$,
this is the statement of Propositions 5 and 6 in \cite{pel14}.
With obvious modifications, their proofs can be adapted to prove Proposition \ref{pro:ext_prob_many} in the case that $\lpr$ is $2$-coherent or centered $2$-convex.
\end{proof}
Proposition \ref{pro:ext_prob_many} ensures that the lower GN-extension determines
the $2$-coherent or $2$-convex natural extensions too.
This is an extremely useful result,
because computing the lower GN-extension is rather simple,
as it only requires to detect $(C|D)_*$.
This is done by means of the first of \eqref{eq:max_cond_event} simply exploiting the logical relationships between the atoms of $\prt$ and both $C$ and $D$.
In the hypotheses of this section,
just the logical relationships allowing us to identify the inner event
$(C|D)_*$ matter to find the extensions we are concerned with,
irrespective of the kind of consistency of $\lpr$ in $\asetc$.
If in particular $\lpr$ is coherent on $\asetc$,
then the four extensions coincide:
\begin{proposition}[Sandwich Theorem]
\label{pro:sandwich_thm}
Let $\eset$ be an arbitrary set of conditional events.
If $\lpr$ is coherent on $\asetc$, then
$\LE(C|D)=\LEC(C|D)=\LED(C|D)=\LEDC(C|D)(=\lpr((C|D)_*))$, $\forall C|D\in\eset$.
\end{proposition}
\begin{proof}
Nearly obvious:
coherence of $\lpr$ ensures that all extensions are finite
by Proposition~\ref{pro:finiteness_ne};
in general they are ranked by \eqref{eq:order_ne},
but since by Proposition~\ref{pro:ext_prob_many} the same procedure is applied to find them all,
they actually coincide,
and are identical to the lower GN-extension.
\end{proof}
Proposition \ref{pro:sandwich_thm} suggests an interesting perspective for comparing $\LE$, $\LEC$, $\LED$, $\LEDC$:
investigate under what conditions they may coincide.
We shall see that they still do so
when $\lpr$ is assessed in the larger environment considered in Section \ref{sec:extension_cond_low_prev},
while they generally do not in the different,
although rather general too, setting in Section \ref{sec:extension_lprob_to_lprev}.

\section{Extending conditional lower previsions}
\label{sec:extension_cond_low_prev}
In this section we suppose that $\dset=\lcond$.
This is an important special case,
because of the closure properties of $\lcond$,
and since $\lcond$ generalises a linear space
(cf. Definition \ref{def:def_dlin}).
On our way to establish when more natural extensions may coincide,
we preliminarily tackle two issues:
\begin{itemize}
	\item[i)]
	derive an alternative expression for the (coherent) natural extension of a coherent lower prevision;
	\item[ii)]
	discuss how to hedge possibily non-finite extensions.
\end{itemize}
As for i),
the following proposition was proven in \cite[Proposition 2]{pel17}.
%We report a proof for the sake of completeness.
\begin{proposition}
	\label{pro:alt_nat_ext}
	Let $\lpr$ be coherent on $\lcond$ and $Z|A$ a conditional gamble.
	Then, defining 
	\begin{eqnarray}
	\label{eq:alt_LE}
	\begin{array}{ll}
	L_1(Z|A)&=\{\alpha:\sup\{BX-A(Z-\alpha)|A\vee B\}<0,\\
	&\mbox{for some } X\in\xset, B\in\bset, \mbox{ with } \lpr(X|B)=0 \mbox{ \emph{if} } B\neq\varnothing\},
	\end{array}
	\end{eqnarray}
	$L_1(Z|A)=L(Z|A)$ and
	the natural extension of $\lpr$ on $Z|A$ is
	\begin{eqnarray}
	\label{eq:alt_nat_ext}
	\LE(Z|A)=\sup L_1(Z|A).
	\end{eqnarray}
\end{proposition}
As a corollary to Proposition \ref{pro:alt_nat_ext},
we may derive a first equality for natural extensions:
\begin{corollary}
\label{cor:equality_e_e2}
If $\lpr$ is coherent on $\lcond$,
then $\LE(Z|A)=\LED(Z|A)$, $\forall Z|A$.
\end{corollary}
\begin{proof}
It holds that $L_1(Z|A)\subseteq L_2(Z|A)$:
let for this $\alpha\in L_1(Z|A)$.
Then $\sup\{BX-A(Z-\alpha)|A\vee B\}<0$,
for some $X\in\xset$, $B\in\bset$.
If $B=\varnothing$, $\alpha\in L_2(Z|A)$ (case $\Ss=\varnothing$).
If $B\neq\varnothing$,
then $\lpr(X|B)=0$ and writing the supremum as
$\sup\{B(X-\lpr(X|B))-A(Z-\alpha)|A\vee B\}<0$
it appears that again $\alpha\in \LD(Z|A)$.

Also, $\LD(Z|A)\subseteq L(Z|A)$: this is due to Definition \ref{def_nat_ext} a), b).

Hence, $L_1(Z|A)\subseteq L_2(Z|A)\subseteq L(Z|A)$, in general.
By Proposition \ref{pro:alt_nat_ext},
if $\lpr$ is coherent on $\lcond$, then $L_1(Z|A)=L_2(Z|A)=L(Z|A)$.
Taking the suprema in $L(Z|A)$ and $L_2(Z|A)$, it ensues that  $\LE(Z|A)=\LED(Z|A), \forall Z|A$.
\end{proof}
While \eqref{eq:alt_nat_ext} supplies a new alternative expression for $\LE(Z|A)$,
observe that it boils down to a known result in the unconditional case,
formally obtained putting $\bset=\{\Omega,\varnothing\}$,
$A=\Omega$ in Proposition \ref{pro:alt_nat_ext}.
\begin{corollary}
\label{co:Walley_unconditional}
If $\lpr$ is coherent on a linear space $\xset$ containing real constants and $Z$ is an arbitrary gamble,
then
$\LE(Z)=\sup(L(Z|\Omega))$ reduces to $\sup\{\lpr(Y):Y\leq Z, Y\in\xset\}$.
%\begin{eqnarray}
%\label{eq:Walley_unconditional}
%\LE(Z)=\sup\{\lpr(X):X\leq Z, X\in\xset\}.
%\end{eqnarray}
\end{corollary}
\begin{proof}
We prove that $\LE(Z)=\sup(L(Z|\Omega))$ can be directly obtained from
\eqref{eq:alt_LE} and \eqref{eq:alt_nat_ext}.
Apply for this Proposition \ref{pro:alt_nat_ext} with $\bset=\{\Omega,\varnothing\}$, $A=\Omega$.
Then,
\begin{align*}
	L(Z|\Omega)&=L_1(Z|\Omega)\\
	&=\{\alpha:\sup(X+\alpha-Z)<0, \mbox{ for some } X\in\xset,
	\mbox{ with }\lpr(X)=0\}\\
	&=\{\lpr(Y):\sup(Y-Z)<0, Y\in\xset\}
\end{align*}
(at the last equality put $Y=X+\alpha$ and recall \eqref{eq:trans} to get $\lpr(Y)=\alpha$).
Noting that $\forall Y\in\xset$ such that $Y\leq Z$, $\forall \epsilon>0$ it is $Y-\epsilon\in\xset$, $\sup((Y-\epsilon)-Z)<0$ and that
$\sup_{\epsilon>0}\lpr(Y-\epsilon)=\sup_{\epsilon>0}(\lpr(Y)-\epsilon)=\lpr(Y)$,
the thesis follows easily.
\end{proof}
In fact,
Corollary \ref{co:Walley_unconditional} is part of the statement of Corollary 3.1.8 in \cite{wal91}
and points out a straightforward procedure for computing the natural extension when
$\lpr$ is defined on the structured (unconditional) set $\xset$.
Interestingly, equation \eqref{eq:alt_nat_ext} is a generalisation of such a procedure to a conditional environment.

Turning to issue ii) from the beginning of this section,
we are interested in guaranteeing that the various natural extensions considered are finite,
i.e. neither $-\infty$ nor $+\infty$.
Regarding $\LE$ (or $\LED$),
we shall apply Proposition~\ref{pro:finiteness_ne} a).
In the case of $\LEC$ or $\LEDC$,
the condition in Proposition \ref{pro:finiteness_ne} b),
that $\lpr(0|A)=0$ for any
additional $Z|A$, is rather natural.
Yet, a $2$-convex or convex $\lpr$ does not necessarily fulfil it.
In fact,
it may be the case that $0|A\in\lcond$ and $\lpr(0|A)\neq 0$,
which we can avoid by restricting our attention to $\emph{centered}$ $2$-convex or convex previsions.
But even doing so, \emph{as we will},
it may happen that $0|A\notin\lcond$ and,
unlike the case of a coherent or $2$-coherent $\lpr$,
$\lpr(0|A)=0$ is not the unique ($2$--)convex extension of $\lpr$.
If $\dset^*\supset \lcond$ is the generic set we wish to extend $\lpr$ to,
let
\begin{eqnarray}
\label{eq:set_ext}
\nmset=\{0|A:Z|A\in\dset^*\},\
\lcondz=\lcond\cup\nmset
\end{eqnarray}
Then,
it holds that
% \cite{pel05, pel17}:
\begin{proposition}
\label{pro:ext_cent}
Let $\lpr$ be centered $2$-convex (alternatively, centered convex) on $\lcond$.
The extension of $\lpr$ to $\lcondz$ such that
$\lpr(0|A)=0$, $\forall 0|A\in\nmset$ is centered $2$-convex (convex).
\end{proposition}
\begin{proof}
In the case that $\nmset=\{0|A\}$, i.e. $\nmset$ is a singleton,
the statement is proved in \cite[Proposition 7(b)]{pel05} for convexity and
\cite[Proposition 4(b)]{pel16} for $2$-convexity.

Let us now consider the general situation, and assume that $\lpr$ is centered convex on $\lcond$.
We shall check, using Definition \ref{def:consistency} c),
that $\lpr$ is centered convex on $\lcondz$.
Let $\LG$, $\Ss$ be as in Definition \ref{def:consistency} c).
Define
$S_1(\unds)=\bigvee\{B_i: X_i|B_i\in\lcond, s_i>0\}$,
$S_0(\unds)=\bigvee\{B_i :0|B_i\in\nmset\setminus\lcond, s_i>0\}$,
so that $\Ss=S_1(\unds)\vee S_0(\unds)$.
We have to evaluate those gains only where neither $S_1(\unds)=\varnothing$ (or else $\LG=0$, trivially)
nor $S_0(\unds)=\varnothing$ (or else $\LG$ regards only gambles already in $\lcond$,
where convexity holds by assumption).
%For this we have to evaluate the supremum of those gains $\LG|\Ss$ in %Definition \ref{def:consistency} c) where,
%putting $\Ss=S_1(\unds)\vee S_0(\unds)$,
%with $S_1(\unds)=\bigvee\{B_i\in\Ss:X_i|B_i\in\lcond\}$,
%$S_0(\unds)=\bigvee\{B_i\in\Ss:0|B_i\in\nmset\setminus\lcond\}$,
%neither $S_1(\unds)=\varnothing$ (or else $\LG=\varnothing$, trivially)
%nor $S_0(\unds)=\varnothing$ (or else $\LG$ regards only gamb les already in %$\lcond$,
%where convexity holds by assumption).
To better distinguish the two different kinds of gambles in $\LG$,
we use the notation $A_i$, $0|A_i$ instead of $B_i$, $0|B_i$ for any $B_i$ such that $0|B_i\in\nmset\setminus\lcond$.
Thus, the generic gain $\LG$ to be inspected may be written as
\begin{eqnarray}
\label{eq:generic_gain}
\LG=\sum_{i=1}^{r} s_i B_i(X_i-\lpr(X_i|B_i))+\sum_{i=r+1}^{n} s_i A_i(0-0)-g_0,
\end{eqnarray}
with $g_0$ being either $A_0 (0-0)$ or $B_0 (X_0-\lpr(X_0|B_0))$ with $X_0|B_0\in\lcond$ and $0 \leq r \leq n$.
Since $S_1(\unds)\neq\varnothing$,
there is at least one gamble of $\lcond$ in $\LG$,
let it be $X|B$.
Hence $0|B\in\lcond$ too.
Now replace in \eqref{eq:generic_gain} $\sum_{i=r+1}^{n}s_i A_i (0-0)$
with $(\sum_{i=r+1}^{n}s_i)B(0-0)$ and $g_0$ with
%$g_0^\prime$,
%equal to $g_0$ if $g_0=B_0(X_0-\lpr(X_0|B_0))$,
%to $B(0-0)$ otherwise.
%\[
%g_0^\prime=\begin{cases}
%g_0 \quad\quad\quad\mbox{ if } g_0=B_0(X_0-\lpr(X_0|B_0))\\
%B(0-0) \mbox{ otherwise}.
%\end{cases}
%\]
\[
g_0^\prime=\begin{cases}
B(0-0) \mbox{ if } g_0=A_0(0-0)\\
g_0 \quad\quad\quad\mbox{ otherwise}.
\end{cases}
\]
We get 
\begin{eqnarray*}
%\label{eq:generic_gain}
\LG^\prime=\sum_{i=1}^{r} s_i B_i(X_i-\lpr(X_i|B_i))+(\sum_{i=r+1}^{n} s_i) B(0-0)-g_0^\prime.
\end{eqnarray*}
Since $\LG^\prime=\LG$,
we obtain
$\sup(\LG|S_1(\unds)\vee S_0(\unds))=\sup(\LG^\prime|S_1(\unds)\vee S_0(\unds))\geq\sup(\LG^\prime|S_1(\unds))\geq 0$,
where convexity of $\lpr$ on $\lcond$ is exploited at the last inequality.

If $\lpr$ is centered $2$-convex on $\lcond$,
an analogous (simpler) proof applies.
\end{proof}
Proposition \ref{pro:ext_cent} suggests that when extending a centered convex or $2$-convex $\lpr$ from $\lcond$ to $\dset^{*}\supset\lcond$
we could consider first extending it to the set
$\lcondz$ defined in \eqref{eq:set_ext},
putting $\lpr(0|A)=0$, $\forall 0|A \in\nmset$.
This justifies the following:
\begin{definition}
	\label{def:nat_ext_plus}
	Given $\lcond\subset\dset^*$,
	let $\lpr$ be defined on $\lcondz$ with $\lpr(0|A)=0, \forall 0|A\in\lcondz$.
	Then, $\LEC^+$, $\LEDC^+$ are the convex,
	respectively $2$-convex natural extension of $\lpr$ from $\lcondz$ to $\lcondz\cup\dset^{*}(=\dset^{*}\cup\nmset)$.
\end{definition}
While $\LEC^+$ need not always coincide with the convex natural extension 
of $\lpr$ from $\lcond$ to $\lcondz$,
adding such zeroes is instead irrelevant when considering the natural extension or the
$2$-coherent natural extension, in the sense of the following lemma.
\begin{lemma}
\label{lem:add_zeroes}
Assign a coherent $\lpr$ on $\lcond$ and
let $\dset^*\supset\lcond$.
Using the notation 
$L(Z|A)$ as in Definition \ref{def_nat_ext} a)
with $\dset$ replaced by $\lcond$,
we write instead $L^{+}(Z|A)$ when $\dset=\lcondz$.

Then $L(Z|A)=L^{+}(Z|A)$,
and consequently
\begin{align*}
	\LE(Z|A)=\sup L(Z|A)=\sup L^{+}(Z|A), \forall Z|A\in\dset^*.
\end{align*} 

A perfectly analogous statement holds when $\lpr$ is $2$-coherent on $\lcond$,
replacing $L(Z|A)$, $L^+(Z|A)$, $\LE(Z|A)$ with
$\LD(Z|A)$, $\LD^+(Z|A)$, $\LED(Z|A)$.
\end{lemma}
\begin{proof}
We prove the first statement,
the proof of the $2$-coherence statement being essentially identical.

Obviously,
$L(Z|A)\subseteq L^+(Z|A)$.
For the opposite inclusion,
let $\alpha\in L^+(Z|A)$ be such that,
letting
\begin{align*}
	 W=\sum_{i=1}^{m}s_i B_i(X_i-\lpr(X_i|B_i))+\sum_{j=1}^{n}t_j B'_j(0-0)-A(Z-\alpha),
\end{align*}
with $X_i|B_i\in\lcond$, $s_i\geq 0\ (i=1,\ldots,m)$, $0|B^\prime_j\in\nmset$, $t_j\geq 0\ (j=1,\ldots,n)$,
it holds that
$\sup\{W|\bigvee_{s_i\neq 0}B_i\vee\bigvee_{t_{j}\neq 0} B'_j\vee A\}<0$.

Then,
since
\begin{align*}
	 &\sup\{\sum_{i=1}^{m}s_i B_i(X_i-\lpr(X_i|B_i))-A(Z-\alpha)|\bigvee_{s_i\neq 0}B_i\vee A\}=\\
	 &\sup\{W|\bigvee_{s_i\neq 0}B_i\vee A\}\leq
	 \sup\{W|\bigvee_{s_i\neq 0}B_i\vee\bigvee_{t_{j}\neq 0} B'_j\vee A\}<0,
\end{align*}
we conclude that $\alpha\in L(Z|A)$ too.
\end{proof}
Note that $\LEP$ could be defined, analogously to
$\LEC^+$, $\LEDC^+$ in Definition \ref{def:nat_ext_plus},
as the natural extension of $\lpr$, coherent on $\lcondz$.
However, Lemma \ref{lem:add_zeroes}
states that $\LEP(Z|A)=\LE(Z|A)$,
and we may therefore use only the symbol $\LE(Z|A)$.
Similarly with $\LED^+(Z|A)$.

We turn now our attention to establishing whether the different natural extensions considered so far coincide under some conditions, for a lower prevision $\lpr$ on $\lcond$.
For this,
we introduce sequentially the following results:
\begin{itemize}
	\item[$\bullet$] Theorem \ref{thm:ext_coer},
	stating the equalities $\LE(Z|A)=\LEDC^{+}(Z|A)=\LED(Z|A)$ assuming $\lpr$ coherent;
	\item[$\bullet$] Theorem \ref{thm:ext_conv},
	proving the equality $\LEC^{+}(Z|A)=\LEDC^{+}(Z|A)$ under the (weaker)
	assumption of centered convexity (hence also of coherence) for $\lpr$;
	\item[$\bullet$] Theorem \ref{thm:final_sandwich}, which,
	making use of Theorems \ref{thm:ext_coer} and \ref{thm:ext_conv},
	proves the equalities $\LE(Z|A)=\LED(Z|A)=\LEC^{+}(Z|A)=\LEDC^{+}(Z|A)$ when $\lpr$ is coherent.
\end{itemize}
%The next theorem, partly proven in \cite{pel17}, establishes the equality of %$\LE$, $\LED$, $\LEDC^+$ under coherence of $\lpr$.
\begin{theorem}
\label{thm:ext_coer}
Let $\lpr$ be coherent on $\lcond(\subset\dset^*)$.
Then,
$\LE(Z|A)=\LEDC^{+}(Z|A)=\LED(Z|A)$, $\forall Z|A\in\dset^*$.
\end{theorem}
\begin{proof}
The equality $\LE(Z|A)=\LE_{2c}^{+}(Z|A)$ was proven in \cite[Theorem 2]{pel17}.
As for the remaining part of the thesis,
from Lemmas \ref{lem:order_ne} and \ref{lem:add_zeroes} we get
$\LE(Z|A)\geq\LED(Z|A)=\LED^+(Z|A)\geq\LEDC^+(Z|A)$.
Then, since $\LE(Z|A)=\LEDC^+(Z|A)$, also  $\LED(Z|A)=\LE(Z|A)=\LEDC^+(Z|A)$.
\end{proof}
Another result of the same kind is
\begin{theorem}
\label{thm:ext_conv}
Let $\lpr$ be centered convex on $\lcondz$. 
Then,
$\LEC^{+}(Z|A)=\LEDC^{+}(Z|A)$, $\forall Z|A\in\dset^*$.
\end{theorem}
\begin{proof}
Term $L_c^+ (Z|A)$ the set $L_c (Z|A)$
in Definition~\ref{def_nat_ext} c) when $\dset=\lcondz$,
similarly to the definition of $L_{2c}^+ (Z|A)$ in the proof
of Theorem \ref{thm:ext_coer}.
We shall prove that $L_c^+=L_{2c}^+$;
consequently, their suprema $\LEC$, $\LEDC$ coincide too.

The inclusion $L_{2c}^+(Z|A)\subseteq L_{c}^+(Z|A)$ is obvious,
hence it only remains to prove that $L_{c}^+(Z|A)\subseteq L_{2c}^+(Z|A)$.
Let then $\alpha\in L_{c}^+(Z|A)$.
This means that there exists some
\begin{eqnarray}
\label{eq:w}
W=\sum_{i=1}^{m} s_i B_i(X_i-\lpr(X_i|B_i))+\sum_{j=1}^{n} t_j B'_j(0-\lpr(0|B'_j))-A(Z-\alpha)
\end{eqnarray}
with $X_i|B_i\in\lcond$, $s_i\geq 0$, $i=1,\ldots,m$, $0|B'_j\in\nmset$, $t_j\geq 0$, $j=1,\ldots n$, 
$\sum_{i=1}^{m} s_i+\sum_{j=1}^{n} t_j=1$,
such that,
defining $\Ss=\vee_{i=1}^{m}\{B_i:s_i\neq 0\}$, $\St=\vee_{j=1}^{n}\{B'_j:t_j\neq 0\}$,
\begin{eqnarray}
\label{eq:supwltz}
\sup\{W|\Ss\vee\St\vee A\}<0.
\end{eqnarray}
We distinguish two situations:
\begin{itemize}
\item[a)]
$0<\sum_{i=1}^{m} s_i=1-s\ (s\in [0,1[)$, hence $\Ss\neq\varnothing$.

Since $\lpr$ is convex on $\lcondz$, using $\sum_{i=1}^{m} s_i+s=1$,
by (A6) and then (A4),
Theorem \ref{thm:characterisation} c)
we get
\begin{align*}
&\lpr(\sum_{i=1}^{m} s_i B_i(X_i-\lpr(X_i|B_i))+s B'_1(0-\lpr(0|B'_1))|\Ss)\geq\\
&\sum_{i=1}^{m} s_i \lpr(B_i(X_i-\lpr(X_i|B_i))|\Ss)+s \lpr(B'_1(0-0)|\Ss)=0.
\end{align*}
%$\lpr(\sum_{i=1}^{m} s_i B_i(X_i-\lpr(X_i|B_i))+s B'_1(0-\lpr(0|B'_1))|\Ss)\geq
%\sum_{i=1}^{m} s_i \lpr(B_i(X_i-\lpr(X_i|B_i))|\Ss)+s \lpr(B'_1(0-0)|\Ss)$
%$=0$.
From this inequality, it follows
\begin{align*}
W^\prime=&\Ss[\sum_{i=1}^{m} s_i B_i(X_i-\lpr(X_i|B_i))-
\lpr(\sum_{i=1}^{m} s_i B_i(X_i-\lpr(X_i|B_i))|\Ss)]\\
&-A(Z-\alpha)\leq\Ss[\sum_{i=1}^{m} s_i B_i(X_i-\lpr(X_i|B_i))]-A(Z-\alpha)=W.
\end{align*}
%$W^\prime=\Ss[\sum_{i=1}^{m} s_i B_i(X_i-\lpr(X_i|B_i))-
%\lpr(\sum_{i=1}^{m} s_i B_i(X_i-\lpr(X_i|B_i))|\Ss)]
%-A(Z-\alpha)\leq\Ss[\sum_{i=1}^{m} s_i B_i(X_i-\lpr(X_i|B_i))]-A(Z-\alpha)=W$.
Hence,
there exists $X|B\in\lcond$ ($X=\sum_{i=1}^{m} s_i B_i(X_i-\lpr(X_i|B_i))$, $B=\Ss$)
such that
$W^\prime = B(X-\lpr(X|B))-A(Z-\alpha)\leq W$,
and that
\begin{align*}
\sup\{W^\prime|B\vee A\}&\leq\sup\{W^\prime|\Ss\vee\St\vee A\}\\
&\leq\sup\{W|\Ss\vee\St\vee A\}<0
\end{align*}
%$\sup\{W^\prime|B\vee A\}\leq\sup\{W^\prime|\Ss\vee\St\vee %A\}\leq\sup\{W|\Ss\vee\St\vee A\}<0$,
recalling \eqref{eq:supwltz} for the strict inequality.
Therefore $\alpha\in L_{2c}^{+}(Z|A)$.
\item[b)]
$\sum_{i=1}^{m}s_{i}=0$, hence $\Ss=\varnothing$.

From \eqref{eq:w}, the gamble
$W=\sum_{j=1}^{n}t_j B'_j(0-0)-A(Z-\alpha)=-A(Z-\alpha)$,
with $\sum_{j=1}^{n}t_j=1$, $t_j\geq 0$, $j=1,\ldots,n$
is such that $\sup\{W|\St\vee A\}<0$.
Since $0|A\in\lcondz$ and noting that $A(0-\lpr(0|A))-A(Z-\alpha)=W$,
it holds that $\sup\{W|A\}\leq\sup\{W|\St\vee A\}<0$.
Hence again $\alpha\in L_{2c}^{+}(Z|A)$.
\end{itemize}
\end{proof}
Finally,
we can now establish a sandwich theorem:
\begin{theorem}[\emph{Sandwich Theorem}]
\label{thm:final_sandwich}
Let $\lpr$ be coherent on $\lcond$.
Then $\LE(Z|A)$
$=\LED(Z|A)=\LEC^+(Z|A)=\LEDC^+(Z|A), \forall Z|A\in\dset^*$.
\end{theorem}
\begin{proof}
	Because of Theorem \ref{thm:ext_coer},
	$\LE(Z|A)=\LED(Z|A)=\LEDC^+(Z|A)$,
	while $\LEDC^+(Z|A)$
	$=\LEC^+(Z|A)$ is due to Theorem \ref{thm:ext_conv}.
\end{proof}
The Sandwich Theorem ensures that the simpler $2$-convex natural extension may be enough to compute the natural extension,
the convex natural extension or even the $2$-coherent natural extension,
in the special case that the starting set is $\lcond$.
This seems to suggest that
only the rather weak properties of (centered) $2$-convexity really matter and need to be checked when looking for a least-committal coherent extension,
if $\lpr$ is initially assessed on a structured enough set and already coherent there.
Structured enough roughly means that the initial set has certain closure properties but also,
as we shall see in the next section,
that it is not `too poor' compared to the new environment we are extending the given uncertainty measure onto.

\section{Extending lower probabilities to lower previsions}
\label{sec:extension_lprob_to_lprev}
In this section we investigate a further,
different extension problem.
The starting environment is now the set $\asetpa$ (cf. Definition \ref{def:special_structures});
$\lpr$ is a lower probability defined on $\asetpa$.
We wish to extend $\lpr$ to $\lset$,
the set of all gambles on $\prt$.

Whilst the initial uncertainty assignment $\lpr$ is still `full' in the sense
that it evaluates all the events depending on $\prt$,
we are operating in an unconditional setting now and,
what is more important,
the extension introduces lower previsions in a pre-existing world of lower probabilities only.

In the next Sections \ref{subsec:2_coherent_ne} and \ref{subsec:2_convex_ne},
simplified expressions for the $2$-coherent and $2$-convex natural extensions are obtained.
They are discussed in Section \ref{subsec:more_on_ne},
where a further variant for computing $\LEDC$ is derived.

\subsection{The $2$-coherent natural extension}
\label{subsec:2_coherent_ne}
Suppose throughout this section that the starting $\lpr$ is $2$-coherent on $\asetpa$.
Preliminarily,
observe that,
given $Z\in\lset$,
Definition \ref{def_nat_ext} b) specialises to
\begin{eqnarray*}
	\LED(Z)=\sup\LD(Z)
\end{eqnarray*}
where
\begin{align}
\label{eq:LD}
\begin{split}
\LD(Z)=\{&\alpha:\sup\{s_1(E-\lpr(E))-(Z-\alpha)\}<0,\\
         &\mbox{ for some } E\in\asetpa, s_1\geq 0\}.
\end{split}
\end{align}
The special structure of $\LD(Z)$ in \eqref{eq:LD} permits the derivation of alternative expressions for $\LED$:
\begin{proposition}
	\label{pro:extending_2coer_to_Z}
Let $\lpr$ be $2$-coherent on $\asetpa$ and $Z\in\lset$.
Define
$\asetz=\{E\in\asetpa\setminus\{\varnothing,\Omega\}: \inf(Z|E)>\inf(Z|\nega{E})\}$,
$\infa=\{\inf(Z|E):E\in\asetz\}$, $F_z=(Z\geq z)$,
$\forall z\in\rset$.
Then,
\begin{align}
\label{eq:2cohneexpr1}
\LED(Z)&=\max\{\inf Z,\sup\{\inf(Z|E)\lpr(E)+\inf(Z|\nega{E})(1-\lpr(E)):E\in\asetz\}\}\\
\label{eq:2cohneexpr2}
&=\max\{\inf Z,\sup\{\inf(Z|F_z)\lpr(F_z)+\inf(Z|\nega{F_z})(1-\lpr(F_z)):z\in \infa\}\}.
\end{align}
\end{proposition}
\begin{proof}
	Preliminarily,
	note that $\asetz$ is empty iff $Z$ is constant, $Z=c\in\rset$.
	In such a case,
	the suprema in \eqref{eq:2cohneexpr1}, \eqref{eq:2cohneexpr2} are
	$-\infty$ and both formulae give $\LED(Z)=\inf Z = c$.
	Assume then that $\asetz$ is non-empty.
	
	Let us now prove \eqref{eq:2cohneexpr1}.
	Define, for any $E\in\asetpa$,
	\begin{align*}
		%\label{eq:LE2}
		\LD^E(Z)&=\{\alpha:\sup\{s_1(E-\lpr(E))-(Z-\alpha)\}<0,
		\mbox{ for some } s_1\geq 0\},\\
		\ale&=\sup\LD^E(Z).
	\end{align*}
	Recalling \eqref{eq:LD}, we may write $\LD(Z)$ as
	\begin{align}
		\label{eq:altexprLD}
		\LD(Z)=\bigcup_{E\in\asetpa}\LD^E(Z),
	\end{align}
	so that
	\begin{align}
		\label{eq:cupLD}
		\LED(Z)=\sup\{\ale:E\in\asetpa\}\}.
	\end{align} 
	We distinguish two cases:
	\begin{itemize}
		\item[1)]
		If $E\in\{\varnothing, \Omega\}$,
		then,
		recalling that $\lpr(\varnothing)=0$, $\lpr(\Omega)=1$ by $2$-coherence of $\lpr$,
		\begin{align*}
			\LD^\varnothing(Z)&=\LD^\Omega(Z)=\{\alpha:\alpha<\inf Z\}.\\
			\alpha^{\varnothing}&=\alpha^{\Omega}=\inf Z.
		\end{align*}
		\item[2)]
		Let $E\notin\{\varnothing, \Omega\}$.
		
		Since, for any $s_1\geq 0$,
		$\sup\{s_1(E-\lpr(E))-(Z-\alpha)\}<0$ iff
		$\alpha<s_1 \lpr(E)+\inf\{Z-s_1 E\}$,
		we can equivalently write $\LD^E(Z)$ as
		\begin{align}
			\label{eq:altae}
			\begin{split}
				\LD^E(Z)=\{\alpha:\alpha<s_1\lpr(E)+&\min\{\inf(Z|E)-s_1,\inf(Z|\nega{E})\},\\
				&\mbox{ for some } s_1\geq 0\}.
			\end{split}
		\end{align}
		Two situations are possible:
		\begin{itemize}
			\item[a)]
			If $\inf(Z|E)\leq\inf(Z|\nega{E})$,
			then $\inf(Z|E)=\inf Z$ and,
			%$\inf(Z|E)-s_1\leq \inf(Z|\nega{E})$, $\forall s_1\geq 0$. Therefore,
			recalling also that $2$-coherence of $\lpr$ implies $\lpr(E)\in [0,1]$
			(follows from \eqref{eq:intern}),
			we get
			\begin{align*}
				\LD^E(Z)&=\{\alpha:\alpha<s_1\lpr(E)+\inf(Z|E)-s_1, \mbox{ for some } s_1\geq 0\}\\
				&=\{\alpha:\alpha<\inf(Z|E)-s_1(1-\lpr(E)), \mbox{ for some } s_1\geq 0\}\\
				&=\{\alpha:\alpha<\inf Z\}=\LD^{\Omega}(Z),\\
				\ale&=\sup\LD^E(Z)=\inf Z=\alpha^{\Omega}.	
			\end{align*}
			
			\item[b)]
			Let $\inf(Z|E)>\inf(Z|\nega{E})$.
			Recalling \eqref{eq:altae}, define
			\begin{align*}
				L_{2,1}^E(Z)=&\{\alpha:\alpha<\inf(Z|E)-s_1(1-\lpr(E)),\\
				&\mbox{ for some } s_1\geq \inf(Z|E)-\inf(Z|\nega{E})\}\\
				L_{2,2}^E(Z)=&\{\alpha:\alpha<s_1\lpr(E)+\inf(Z|\nega{E}),\\
				&\mbox{ for some } s_1\in[0,\inf(Z|E)-\inf(Z|\nega{E})]\}.
			\end{align*}
			It holds that $\LD^E(Z)=L_{2,1}^E(Z)\ \cup\ L_{2,2}^E(Z)$,
			%defining $\ale_1=\sup L_{2,1}^E(Z)$, $\ale_2=\sup L_{2,2}^E(Z)$,
			$\ale=\max\{\sup L_{2,1}^E(Z),\\ \sup L_{2,2}^E(Z)\}$.
			Noting that
			both $L_{2,1}^E(Z)$ and $L_{2,2}^E(Z)$ are intervals and their suprema coincide and
			are achieved by putting $s_1=\inf(Z|E)-\inf(Z|\nega{E})$,
			we get
			\begin{align}
			\label{eq:LE2_less}
			\nonumber
			    \LD^E(Z)&=L_{2,1}^E(Z)=L_{2,2}^E(Z)\\
			    %&=(-\infty,\inf(Z|E)\lpr(E)+\inf(Z|\nega{E})(1-\lpr(E)))\\
			 %&=\{\alpha:\alpha<\inf(Z|E)\lpr(E)+\inf(Z|\nega{E})(1-\lpr(E))\}\\
				\ale&=\inf(Z|E)\lpr(E)+\inf(Z|\nega{E})(1-\lpr(E)).
			\end{align}
		\end{itemize}
		We conclude from \eqref{eq:altexprLD} that
		\begin{align}
		\label{eq:altexprLD2}
		\LD(Z)=\bigcup_{E\in\asetz}\LD^E(Z)\cup\LD^{\Omega}(Z)
		\end{align}
		that is, $\Omega$ and the events in $\asetz$ suffice
		to compute $\LED(Z)$ by means of \eqref{eq:cupLD}.
		This gives \eqref{eq:2cohneexpr1}.
	\end{itemize}
	
	To prove \eqref{eq:2cohneexpr2},
	we show that, defining
	$$\ze=\inf(Z|E),\forall E\in\asetpa\setminus\{\varnothing,\Omega\},$$
	hence
	$F_{\ze}=(Z\geq \ze)$, it holds that
	\begin{eqnarray}
	\label{eq:subLD}
	\forall E\in\asetz, F_{\ze}\in\asetz \mbox{ and }
	\LD^E(Z)\subseteq\LD^{F_{\ze}}(Z).
	\end{eqnarray}
	We preliminarily note that:
	\begin{itemize}
		\item[i)]
		$\inf(Z|F_{\ze})=\inf(Z|E)$.
		
		In fact,
		$\inf(Z|F_{\ze})\leq\inf(Z|E)$ because $E\Rightarrow F_{\ze}$,
		and $\inf(Z|F_{\ze})=\inf(Z|Z\geq\inf(Z|E))\geq\inf(Z|E)$
		by definition of the infimum.
		\item[ii)]
		$\lpr(E)\leq\lpr(F_{\ze})$,
		again because $E\Rightarrow F_{\ze}$ (cf. \eqref{eq:monot}).
		\item[iii)]
		If $\inf(Z|E)>\inf(Z|\nega{E})$,
		then $\inf(Z|\nega{F_{\ze}})=\inf(Z|\nega{E})$.
		
		In fact,
		$\inf(Z|E)>\inf(Z|\nega{E})$ implies
		$\nega{F_{\ze}}=(Z<\inf(Z|E))\neq\varnothing$.
		Then,
		$\inf(Z|\nega{F_{\ze}})=\inf Z=\min\{\inf(Z|E),\inf(Z|\nega{E})\}=\inf(Z|\nega{E})$.
		\end{itemize}
	Let $E\in\asetz$, i.e. $\inf(Z|E)>\inf(Z|\nega{E})$. Then,
	\begin{itemize}
		\item[a)]
		$F_{\ze}\in\asetz$.
		
		In fact, 
		$\inf(Z|F_{\ze})=\inf(Z|E)>\inf(Z|\nega{E})=\inf(Z|\nega{F_{\ze}})$,
		by i) and iii).

		\item[b)]
		$\LD^E(Z)\subseteq\LD^{F_{\ze}}(Z)$.
		
		Let $\alpha\in\LD^E(Z)$.
		Then, by \eqref{eq:LE2_less},
		\begin{align}
		\nonumber
	     \alpha&<\inf(Z|E)\lpr(E)+\inf(Z|\nega{E})(1-\lpr(E))\\
	     \nonumber
	     &=\lpr(E)(\inf(Z|E)-\inf(Z|\nega{E}))+\inf(Z|\nega{E})\\
	     \nonumber
	     &\leq\lpr(F_{\ze})(\inf(Z|F_{\ze})-\inf(Z|\nega{F_{\ze}}))+\inf(Z|\nega{F_{\ze}})\\
	     \nonumber
	     &=\inf(Z|F_{\ze})\lpr(F_{\ze})+\inf(Z|\nega{F_{\ze}})(1-\lpr(F_{\ze})),
		\end{align}
		using i), ii), iii) and $\lpr(E)\geq 0$ at the weak inequality.
		Hence $\alpha\in\LD^{F_{\ze}}(Z)$, again by \eqref{eq:LE2_less}.
		
	\end{itemize}

	Therefore, \eqref{eq:subLD} holds.
	This implies that all distinct events $E\in\asetz$ such that
	$\inf(Z|E)$ has a common value,
	equal to $\inf(Z|F_{\ze})$,
	might be replaced by the single event $F_{\ze}$ in \eqref{eq:altexprLD2}.
	Hence, we may simplify \eqref{eq:altexprLD2}:
	\begin{align*}
		%\label{eq:altexprLD2}
		\LD(Z)=\bigcup_{z\in \infa}\LD^{F_{z}}(Z)\cup\LD^{\Omega}(Z).
	\end{align*}
	%Note that \eqref{eq:altexprLD2} lets us select $F_z$ as a representative %of \emph{every} $E\in\asetz$ such that $\inf(Z|E)=\inf(Z|F_{z})$ in the %expression for $\LD(Z)$.
	This gives \eqref{eq:2cohneexpr2}.
\end{proof}
%
%Let us comment on Proposition \ref{pro:extending_2coer_to_Z}.
%Firstly,
%note that its proof is essentially unchanged if $\lpr$ is defined on %$\dset\subseteq\asetpa$, $\dset\supseteq\mathcal{F}=\{F_{z}: z\in\rset\}$ %(just replace $\asetpa$ with $\dset$ everywhere).
%Therefore,
%Proposition \ref{pro:extending_2coer_to_Z} generalises to
%\begin{corollary}
%\label{cor:ext_2coer_Z}
%Let $\lpr$ be $2$-coherent on $\dset$, such that %$\mathcal{F}\subseteq\dset\subseteq\asetpa$.
%Then, $\forall Z\in\lset$,
%its natural extension $\LED(Z)$ is given by either of \eqref{eq:2cohneexpr1} %and \eqref{eq:2cohneexpr2}.
%\end{corollary}
%
%RIPRENDE LA PARTE GIA' PREVISTA PER LA 2-COERENZA MA IL DISCORSO NON PUO' %STARE PIU' QUI.
%
%Even though the expression in \eqref{eq:2cohneexpr1} represents a %simplification over the general Definition \ref{def_nat_ext} b) of %$2$-coherent natural extension,
%nevertheless the reduction of complexity achieved by \eqref{eq:2convneexpr1} %is definitely larger.
%In fact,
%let $\prt$ be finite and made up of $n$ atoms.
%Then we may have to compute $\max(Z|E)-\max(Z|\nega{E})$ for $2^{n-1}-1$ %couples $(E,\nega{E})$,
%and to select the maximum among the $2^{n-1}-1$ terms
%$\min(Z|E')\lpr(E')+\min(Z|\nega{E'})(1-\lpr(E'))$,
%with $E'=E$ or $E'=\nega{E}$.
%By contrast,
%we have already noticed that the $2$-convex natural extension requires at most %$n+1$ comparisons.
%Let illustrate the two procedures with an example.

\subsection{The $2$-convex natural extension}
\label{subsec:2_convex_ne}
Suppose now that $\lpr$ is $2$-convex on $\asetpa$.
Again,
simplified expressions are available for the $2$-convex natural extension $\LEDC(Z)$, $Z\in\lset$.
\begin{proposition}
	\label{pro:extending_2conv_to_Z}
	Let $\lpr$ be centered $2$-convex on $\asetpa$ and $Z\in\lset$.
	Define
	$F_z=(Z\geq z)$,
	$\forall z\in\rset$,
	$\bsetz=\{E\in\asetpa\setminus\{\varnothing,\Omega\}: \inf(Z|E)>\inf(Z|\nega{E})-1\}$,
	$\infb=\{\inf(Z|E):E\in\bsetz\}$.
	Then,
	\begin{align}
	\label{eq:2convneexpr1}
	\LEDC(Z)&=\max\{\inf Z,\sup\{\lpr(E)+\min\{\inf(Z|E)-1,\inf(Z|\nega{E})\}: E\in\bsetz\}\}\\
	\label{eq:2convneexpr2}
	&=\max\{\inf Z,\sup\{\lpr(F_z)+\min\{\inf(Z|F_z)-1,\inf(Z|\nega{F_z})\}: z\in \infb\}\}.
	\end{align}
\end{proposition}
\begin{proof}
	In its first part, the proof is analogous to that of Proposition \ref{pro:extending_2coer_to_Z}
	(but note that $\forall Z$, $\bsetz$ is not empty if $\prt\neq\{\Omega\}$,
	unlike $\asetz$ in Proposition \ref{pro:extending_2coer_to_Z}).
	
	Defining
	\begin{align*}
	%\label{eq:LE2C}
	\LDC^E(Z)&=\{\alpha:\sup\{(E-\lpr(E))-(Z-\alpha)\}<0\},\\
	\ale&=\sup\LDC^E(Z),
	%\\
	%\ale&=\sup\LD^E(Z).
	\end{align*}
	we have 
	\begin{align}
	\label{eq:altexprLDc}
	\LDC(Z)=\bigcup_{E\in\asetpa}\LDC^E(Z).
	\end{align}
	Then, similarly to Proposition \ref{pro:extending_2coer_to_Z}, we may compute
		\begin{align}
		\label{eq:cupLDC}
		\LEDC(Z)=\sup\{\ale:E\in\asetpa\}.
		\end{align}
		%	Recalling \eqref{eq:LD}, we may write $\LD(Z)$ as
%	\begin{align}
%	\label{eq:altexprLD}
%	\LD(Z)=\bigcup_{E\in\asetpa}\LD^E(Z),
%	\end{align}
%	so that
 	To obtain \eqref{eq:2convneexpr1} from \eqref{eq:cupLDC}, we distinguish two cases:
	\begin{itemize}
		\item[1)]
		If $E\in\{\varnothing, \Omega\}$,
		then,
		since $\lpr(\varnothing)=0$, $\lpr(\Omega)=1$ by centered $2$-convexity of $\lpr$ (cf. \eqref{eq:intern}),
		%\begin{align*}
		%\LDC^\varnothing(Z)=\LDC^\Omega(Z)=\{\alpha:\alpha<\inf Z\},
		%\end{align*}
		we get
		\begin{align}
		\label{eq:Lzero}
		    \LDC^\varnothing(Z)&=\LDC^\Omega(Z)=\{\alpha:\alpha<\inf Z\}\\
			%\sup\LDC^\varnothing(Z)=\sup\LDC^\Omega(Z)=\inf Z.
			\label{eq:Lzero_bis}
			\alpha^{\varnothing}&=\alpha^{\Omega}=\inf Z
		\end{align} 
		\item[2)]
		If $E\notin\{\varnothing, \Omega\}$, then,
		since $\sup\{(E-\lpr(E))-(Z-\alpha)\}<0$ iff
		$\alpha<\lpr(E)+\inf\{Z-E\}$
		iff $\alpha<\lpr(E)+\min\{\inf(Z|E)-1,\inf(Z|\nega{E})\}$,
		%Therefore, we can equivalently write $\LDC^E(Z)$ as
		%\begin{align}
		%\label{eq:altae}
		%\begin{split}
		%\LD^E(Z)=\{\alpha:\alpha<s_1\lpr(E)+&\min\{\inf(Z|E)-s_1,\inf(Z|\nega{E%})\},\\
		%&\mbox{ for some } s_1\geq 0\}
		%\end{split}
		%\end{align}
		\begin{align}
		\label{eq:LDC_expression}
		    \LDC^{E}(Z)&=\{\alpha:\alpha<\lpr(E)+\min\{\inf(Z|E)-1,\inf(Z|\nega{E})\}\\
		 \label{eq:LDC_expression_bis}
			\ale&=\lpr(E)+\min\{\inf(Z|E)-1,\inf(Z|\nega{E})\}
		\end{align} 
	\end{itemize}
We note that any set $\LDC^{E}(Z)$ such that $\inf(Z|E)\leq\inf(Z|\nega{E})-1$ may be dropped from the union in \eqref{eq:altexprLDc},
%for the purpose of computing the supremum in \eqref{eq:cupLDC},
since then $\LDC^{E}(Z)\subseteq\LDC^{\nega{E}}(Z)$.
%In fact, in this case

In fact, let $\inf(Z|E)\leq\inf(Z|\nega{E})-1$.
Then, by \eqref{eq:LDC_expression},
\begin{align*}
	\LDC^{E}(Z)&=\{\alpha:\alpha<\lpr(E)+\inf(Z|E)-1\}\\
	\LDC^{\nega{E}}(Z)&=\{\alpha:\alpha<\lpr(\nega{E})+\inf(Z|E)\}.
\end{align*}
Take $\alpha\in\LDC^{E}(Z)$.
Recalling that $\lpr(E)\leq 1$ and $\lpr(\nega{E})\geq 0$ by $2$-convexity, we get
$\alpha<\lpr(E)+\inf(Z|E)-1\leq\inf(Z|E)\leq\lpr(\nega{E})+\inf(Z|E)$,
i.e. $\alpha\in\LDC^{\nega{E}}(Z)$.

The above argument simplifies \eqref{eq:altexprLDc}:
\begin{align*}
\LDC(Z)=\bigcup_{E\in\bset^{Z}}\LDC^E(Z)\cup\LDC^{\Omega}(Z).
\end{align*}

This gives \eqref{eq:2convneexpr1} from \eqref{eq:cupLDC}.
	
	To prove \eqref{eq:2convneexpr2},
	define,
	as in Proposition \ref{pro:extending_2coer_to_Z},
	$\ze=\inf(Z|E),\forall E\in\asetpa\setminus\{\varnothing,\Omega\}$,
	so that $F_{\ze}=(Z\geq \ze)$.

    We prove that
	\begin{eqnarray}
	\label{eq:subLDC}
	\LDC^E(Z)\subseteq\LDC^{F_{\ze}}(Z),\ \forall E\in\asetpa\setminus\{\varnothing,\Omega\},
	\end{eqnarray}
	with $F_{\ze}=\Omega$ or $F_{\ze}\in\mathcal{B}^{Z}$.
	
	For this, take $E\in\asetpa\setminus\{\varnothing,\Omega\}$, $\alpha\in \LDC^E(Z)$,
	i.e. such that $\alpha$ satisfies \eqref{eq:LDC_expression}.

	We prove that $\alpha\in\LDC^{F_{\ze}}(Z)$,
	i.e. that $\alpha$ satisfies either the inequality
	\begin{eqnarray}
	\label{eq:alpha_fzebis_c}
	\alpha<\inf Z,
	\end{eqnarray}
	when $F_{\ze}=\Omega$, or the inequality
	\begin{eqnarray}
	\label{eq:alpha_fze_c}
	\alpha<\lpr(F_{\ze})+\min\{\inf(Z|F_{\ze})-1,\inf(Z|\nega{F_{\ze}})\},
	\end{eqnarray} 
	when $F_{\ze}\neq\Omega$.	
	
	We already know that $\inf(Z|F_{\ze})=\inf(Z|E)$, $\lpr(E)\leq\lpr(F_{\ze})$
	and that 
	$\inf(Z|E)>\inf(Z|\nega{E})$ implies $\inf(Z|\nega{F_{\ze}})=\inf(Z|\nega{E})$
	from i), ii) and iii) in the proof of Proposition~\ref{pro:extending_2coer_to_Z}.
    Two alternatives occur:
	\begin{itemize}
		\item[1)]
		If \emph{$\inf(Z|E)=\inf Z$},
		then $F_{\ze}=\Omega$.
		
		Consider $\alpha$ satisfying \eqref{eq:LDC_expression}.
		Then, we get
		$\alpha<\lpr(E)+\min\{\inf(Z|E)-1,\inf(Z|\nega{E})\}\leq\lpr(F_{\ze})+\min\{\inf Z-1,\inf(Z|\nega{E})\}=\lpr(\Omega)+\inf Z-1=\inf Z$.
		This proves that $\alpha$ satisfies equation \eqref{eq:alpha_fzebis_c},
		i.e. $\alpha\in\LDC^{\Omega}(Z)$.
		\item[2)]
		If \emph{$\inf(Z|E)>\inf Z$},
		then $\inf (Z|F_{\ze})=\inf(Z|E)>\inf Z=\inf(Z|\nega{E})=\inf (Z|\nega{F_{\ze}})>\inf (Z|\nega{F_{\ze}})-1$.
		This shows that $F_{\ze}\in\bsetz$.
		
		Again, take $\alpha$ satisfying \eqref{eq:LDC_expression}.
		Then, we get immediately $\alpha<\lpr(F_{\ze})+\min\{\inf (Z|E)-1,\inf (Z|\nega{E})\}=\lpr(F_{\ze})+\min\{\inf (Z|F_{\ze})-1,\inf (Z|\nega{F_{\ze}})\}$.
		Hence, $\alpha$ satisfies \eqref{eq:alpha_fze_c},
		i.e. $\alpha\in\LDC^{F_{\ze}}(Z)$.
	\end{itemize}
	We conclude that \eqref{eq:subLDC} holds and that:
	\begin{align}
	\label{eq:altexprLDC2}
	\LDC(Z)=\bigcup_{z\in \infb}\LDC^{F_{z}}(Z)\cup\LDC^{\Omega}(Z).
	\end{align}
This gives \eqref{eq:2convneexpr2}.
	
\end{proof}
%In the rest of this section we may assume that $\lpr$ is defined on $\dset$;
%in the sequel of the paper, however, we shall focus on $\dset=\asetpa$.

\subsection{More on the $2$-coherent and $2$-convex natural extensions}
\label{subsec:more_on_ne}
Let us compare the two alternative expressions supplied for both $\LED$ and $\LEDC$
in Propositions \ref{pro:extending_2coer_to_Z} and \ref{pro:extending_2conv_to_Z}.
Equations
\eqref{eq:2cohneexpr1} and \eqref{eq:2convneexpr1} are of interest in themselves,
as they show that $\LED$ and $\LEDC$ may be expressed in a way rather different from their original definitions,
by involving the events of $\asetpa\setminus\{\varnothing,\Omega\}$ and the infima of $Z$ conditional on such events.
However,
it appears from the proofs of Propositions \ref{pro:extending_2coer_to_Z} and \ref{pro:extending_2conv_to_Z} that formulae \eqref{eq:2cohneexpr2} and \eqref{eq:2convneexpr2} are derived from, respectively,
\eqref{eq:2cohneexpr1} and \eqref{eq:2convneexpr1} by reducing the number of events effectively employed:
event $F_{\ze}=(Z\geq\inf(Z|E))$ is a representative for all events
$E\in\asetpa\setminus\{\varnothing,\Omega\}$ with respect to which $\inf(Z|E)$ takes the same value.
Thus \eqref{eq:2cohneexpr2} and \eqref{eq:2convneexpr2} ensure a further simplification over \eqref{eq:2cohneexpr1} and \eqref{eq:2convneexpr1}.

Formulae \eqref{eq:2cohneexpr2} and \eqref{eq:2convneexpr2} highlight another interesting feature of $\LED(Z)$ and $\LEDC(Z)$:
the only uncertainty evaluation they explicitly depend on is $\lpr(F_z)$,
$z\in\rset$
(more precisely, the range of $z$ may be restricted to either $\infa$ or $\infb$ in Propositions \ref{pro:extending_2coer_to_Z} or \ref{pro:extending_2conv_to_Z} respectively).
Since $\lpr(F_z)=\lpr(Z\geq z)$,
we come to the following conclusion:
\begin{quote}
The values of $\LED(Z)$ and of $\LEDC(Z)$ both depend on the evaluation of (restrictions of) the
\emph{lower decumulative distribution function}\footnote{
	We use this terminology by analogy with the classical definition of decumulative distribution function, also called decreasing distribution function or survival function and often referring to the events $(Z>z)$ rather than $(Z\geq z)$.
}
$\lpr(Z\geq z)$, $z\in\rset$.	
\end{quote}
For the considerations to follow a further simplification for the expression of $\LEDC$ will be useful.
It is stated in the next Proposition.
\begin{proposition}
	\label{pro:extending_2conv_to_Z_bis}
	In the same assumptions of Proposition \ref{pro:extending_2conv_to_Z},
	we have
	%Let $\lpr$ be centered $2$-convex on $\asetpa$ and $Z\in\lset$.
	%Define
	%$F_z=(Z\geq z)$,
	%$\forall z\in\rset$,
	%$\bsetz=\{E\in\asetpa\setminus\{\varnothing,\Omega\}: %\inf(Z|E)>\inf(Z|\nega{E})-1\}$,
	%$I_B=\{\inf(Z|E):E\in\bsetz\}$.
	%Then,
	\begin{align}
		\label{eq:2convneexpr3}
		\LEDC(Z)&=\max\{\sup\{\lpr(F_z)+\inf(Z|F_z)-1: z\in \infb, z<\inf Z+1\},\\
		\nonumber &\quad\quad\quad\  \lpr(F_{\inf Z + 1})+\inf Z\}.
	\end{align}
\end{proposition}
\begin{proof}
Looking at \eqref{eq:2convneexpr2} in Proposition \ref{pro:extending_2conv_to_Z}, we realize that not all $z\in \infb$ are generally needed to obtain $\LDC(Z)$ in \eqref{eq:altexprLDC2}.

In fact, define $\overline{z}=\inf Z+1$, $k=\inf(Z|F_{\overline{z}})$ and
take $z\in \infb$ such that $z\geq\overline{z}$ (if any).
Then, $k\in \infb$ and
$\LDC^{F_{z}}\subseteq\LDC^{F_{\overline{z}}}=\LDC^{F_{k}}$.

To see this, note first that
$z\geq\overline{z}$
implies $F_{z}\neq\Omega$ and
\begin{align}
\label{eq:chain_inf}
\inf(Z|F_{z})=\inf(Z|Z\geq z)\geq z\geq\inf Z+1>\inf Z=\inf(Z|\nega{F_{z}}).
\end{align} 
Hence,
\begin{itemize}
	\item
	$\LDC^{F_{z}}\subseteq\LDC^{F_{\overline{z}}}$.
	
	In fact, from \eqref{eq:LDC_expression} and \eqref{eq:chain_inf},
	we derive
	\begin{align*}
	\LDC^{F_{z}}&=\{\alpha:\alpha<\lpr(F_{z})+\min\{\inf(Z|F_{z})-1,\inf Z\}\}\\
	&=\{\alpha:\alpha<\lpr(F_{z})+\inf Z\}\\
	&\subseteq\{\alpha:\alpha<\lpr(F_{\overline{z}})+\inf Z\}
	=\LDC^{F_{\overline{z}}},
	\end{align*}
	where the inclusion holds because $F_{z}\Rightarrow F_{\overline{z}}$ implies
	$\lpr(F_{z})\leq\lpr(F_{\overline{z}})$.
	\item 
	$F_{\overline{z}}= F_{k}$,
	hence $\LDC^{F_{\overline{z}}}=\LDC^{F_{k}}$, and $k\in \infb$.\footnote{
	This bullet is necessary, because $\overline{z}$ may not be an admissible value for $Z$.
	}
	
	Clearly, $F_{k}\Rightarrow F_{\overline{z}}$.
	For the reverse implication,
	note that for any $\omega$ such that $\omega\Rightarrow F_{\overline{z}}$ it is
	$Z(\omega)\geq\inf(Z|F_{\overline{z}})=k$, hence $\omega\Rightarrow F_{k}$.
	
	To prove that $k\in \infb$,
	put $z=\overline{z}$ in \eqref{eq:chain_inf}
	to get $\inf(Z|F_{\overline{z}})>\inf(Z|\nega{F_{\overline{z}}})-1$,
	hence $F_{\overline{z}}\in\mathcal{B}^Z$ and
	$k=\inf(Z|F_{\overline{z}})\in \infb$.
\end{itemize}
We conclude that
\eqref{eq:altexprLDC2} simplifies further,
since any $\LDC^{F_{z}}(Z)$ 
with $z>\inf Z + 1$ can be replaced with $\LDC^{F_{\inf Z + 1}}(Z)$:\footnote{
	Note that $\{z\in \infb: z\geq\inf Z +1\}=\varnothing$ implies  $\LDC^{F_{\inf Z + 1}}(Z)=\LDC^{\Omega}(Z)$.
}
\begin{align}
\label{eq:altexprLDC2further}
\LDC(Z)=\bigcup_{z\in \infb: z<\inf Z+1}\LDC^{F_{z}}(Z)\cup\LDC^{F_{\inf Z+1}}(Z)\cup\LDC^{\Omega}(Z).
\end{align}

Finally, in order to get \eqref{eq:2convneexpr3}
we have to compute the suprema of the sets $\LDC^{F_{\inf Z+1}}(Z)$ and $\LDC^{F_{z}}(Z)$, with $z\in \infb$, $z<\inf Z +1$.

Clearly, if $z=\inf Z$, $F_{z}=\Omega$,
hence, from \eqref{eq:Lzero_bis}, $	\sup\LDC^{F_{z}}=\inf Z$.
%$=\lpr(\Omega)+\inf Z-1$.	

When $z=\inf Z+1$ or $z\in \infb$ and $\inf Z<z<\inf Z+1$,
we get instead
\begin{align}
\label{eq:lf2c_simplified1}
\sup\LDC^{F_{z}}&=\lpr(F_{z})+\inf(Z|F_{z})-1 &\mbox{ if } &\inf Z<z<\inf Z+1,\\
\label{eq:lf2c_simplified2}
\sup\LDC^{F_{z}}&=\lpr(F_{\inf X+1})+\inf(Z) &\mbox{ if } &z=\inf Z +1.
\end{align}
In fact,
\begin{itemize}
	\item
	Let $z\in(\inf Z, \inf Z+1)\cap \infb$.
	Since $z\in \infb$,
	there exists $E\in\mathcal{B}^Z$ such that $z=\inf(Z|E)=\inf(Z|F_z)$.
	This implies
	$\inf Z+1>z=\inf(Z|F_z)>\inf Z=\inf(Z|\nega{F_{z}})$.	
	Hence, \eqref{eq:lf2c_simplified1} follows from \eqref{eq:LDC_expression_bis}.
	\item
	If  $z=\inf Z +1$, then, by \eqref{eq:chain_inf},
	$\inf(Z|F_{\inf Z +1})-1\geq\inf(Z|\nega{F_{\inf Z +1}})=\inf Z$,
	which gives \eqref{eq:lf2c_simplified2},
	using \eqref{eq:LDC_expression_bis} again.
\end{itemize}
Now take the suprema in \eqref{eq:altexprLDC2further}, recalling \eqref{eq:lf2c_simplified1}, \eqref{eq:lf2c_simplified2}.
This gives \eqref{eq:2convneexpr3}.	
\end{proof}
Operationally,
formulae \eqref{eq:2cohneexpr1}, \eqref{eq:2cohneexpr2}, \eqref{eq:2convneexpr1}, \eqref{eq:2convneexpr3} offer simple ways of computing
$\LED$, $\LEDC$
when $\prt$ is finite.
The following example illustrates the procedures.
\begin{example}
	\label{exa:comparison_2c_2c}
	Let $\prt=\{a,b,c,d\}$, and define $\lpr$ on $\asetpa$ as follows:
	\begin{align*}
	&\lpr(a)=0.2, \lpr(b)=0, \lpr(c)=0.3, \lpr(d)=0.1;\\
	&\lpr(a\vee b)=0.3, \lpr(a\vee c)=0.5, \lpr(a\vee d)=0.4;\\
	&\lpr(b\vee c)=0.5, \lpr(b\vee d)=0.3, \lpr(c\vee d)=0.5;\\
	&\lpr(a\vee b\vee c)=0.6, \lpr(a\vee c\vee d)=0.7, \lpr(b\vee c\vee d)=0.7, \lpr(a\vee b\vee d)=0.7;\\
	&\lpr(\varnothing)=0, \lpr(\Omega)=1.
	\end{align*}
	\begin{itemize}
		\item
		$\lpr$ is $2$-coherent on $\asetpa$,
		hence also centered $2$-convex.
		To prove $2$-coherence, note that $\lpr$ satisfies the conditions of Proposition 4 in \cite{pel16bis}:
		$\lpr(A)\leq\lpr(B), \forall A,B\in\asetpa: A\Rightarrow B$, $\lpr(A)+\lpr(\nega{A})\leq 1, \forall A\in\asetpa$,
		$\lpr(\varnothing)=0$, $\lpr(\Omega)=1$.
		\item
		$\lpr$ is not coherent.
		Use the necessary condition for coherence
		$\lpr(A\vee B)\geq\lpr(A)+\lpr(B)$,
		if $A\wedge B=\varnothing$ \cite[Section 2.7.4 (e)]{wal91}
		with $A=a$, $B=b\vee c$:
		$\lpr(a\vee b\vee c)=0.6<\lpr(a)+\lpr(b\vee c)=0.7$.
		%\item
		%$\lpr$ avoids sure loss,
		%since $\lpr\leq P$,
		%with $P$ precise coherent prevision.
		%One such $P$ is obtained from the assignment 
		%$P(a)=P(d)=0.2$, $P(b)=P(c)=0.3$.
	\end{itemize}
	Now consider $Z\in\lset$,
	defined by:
	$Z(a)=-1$, $Z(b)=0$, $Z(c)=1$, $Z(d)=3$.
	\begin{itemize}
		\item
		Let us compute $\LEDC(Z)$,
		using \eqref{eq:2convneexpr3}.
		Here $\infb=\{0,1,3\}$,
		$\inf Z =-1$,
		hence
		$\{z\in \infb:z<\inf Z+1\}=\varnothing$ and
		$F_{\inf Z + 1}=F_{0}=b\vee c\vee d$.
		Therefore,
		\begin{align*}
		\LEDC(Z)&=\max\{\min Z,	\lpr(b\vee c\vee d)+\inf Z\}\\
		%&=\max\{\lpr(\Omega)+\inf(Z|\Omega)-1,\lpr(b\vee c\vee d)+\inf Z\}\\
		&=\max\{-1,0.7+(-1)\}=-0.3.
		\end{align*}
		%\begin{align*}
		%\LEDC(Z)&=\max\{\lpr(F_{-1}+\inf(Z|F_{-1}))-1,
		%\lpr(b\vee c\vee d)+\inf Z\}\\
		%&=\max\{\lpr(\Omega)+\inf(Z|\Omega)-1,\lpr(b\vee c\vee d)+\inf Z\}\\
		%&=\max\{-1,0.7+(-1)\}=-0.3.
		%\end{align*}
		\item
		To compute $\LED$ by means of \eqref{eq:2cohneexpr1},
		we have to check the sign of $\inf(Z|E)-\inf(Z|\nega{E})$ for each $E\in\asetpa-\{\varnothing,\Omega\}$.
		Then we determine the supremum (here maximum) in \eqref{eq:2cohneexpr1}.
		It turns out that it is achieved by considering the event
		$E=c\vee d\in\asetz$.
		We obtain
		\begin{align*}
		\LED(Z)&=\max\{\min Z, \min(Z|c\vee d)\lpr(c\vee d)+\min(Z|a\vee b)(1-\lpr(c\vee d))\}\\
		&=\max\{-1, 1\cdot 0.5+(-1)\cdot (1-0.5)\}=0.
		\end{align*} 
	\end{itemize}
	Note that $\LEDC(Z)=-0.3<\LED(Z)=0$.
\end{example}
\begin{remark}[How vacuous is the $2$-convex natural extension?]
$ $
\newline
Equation \eqref{eq:2convneexpr3} is also relevant in focusing our intuition that the $2$-convex natural extension should intrinsically tend to be vague.
Results like the Sandwich Theorem \ref{thm:final_sandwich} seem to indicate the opposite,
that there are instances such that $\LEDC$ is not so vacuous after all.
However,
this is not the general rule.
In fact,
we immediately obtain the following upper bound for $\LEDC$ from \eqref{eq:2convneexpr3}:
\begin{align}
\label{eq:upper_bound_2conv}
\LEDC(Z)\leq\inf Z + 1, \forall Z\in\lset.
\end{align}
Inequality \eqref{eq:upper_bound_2conv} shows that the $2$-convex natural extension may be nearly vacuous.
This is what happens
if $\sup Z-\inf Z\gg 1$,
no matter whether $\lpr$ is already coherent or at least convex on $\asetpa$:
compared to the range of $Z$, $\LEDC(Z)$ nearly overlaps with $\inf Z$,
its smallest admissible value.
Note that $\LEDC(Z)$ is not necessarily vague when $\sup Z-\inf Z<1$,
but $\LEDC(\lambda Z)$ is so,
if $\lambda>0$ is large enough.
This substantial asymmetry in treating extensions of gambles differing only by a scaling factor is not surprising:
axiom (A5) does not necessarily apply to $2$-convex previsions,
being actually a distinguishing feature between them and the $2$-coherent previsions.
\end{remark}

\subsection{The Choquet integral extension}
\label{subsec:Choquet_extension}
If the given $\lpr$ is $2$-coherent on $\asetpa$,
it is possible to extend it onto $\lset$,
preserving $2$-coherence,
resorting to another extension,
generally dominating the $2$-coherent natural extension.
This is the \emph{Choquet integral extension} $\LC$,
discussed in this section.
The possibility of exploiting $\LC$ relies on the following:
\begin{proposition}
\label{pro:Choquet_2c}
Given $\lpr$ $2$-coherent on $\asetpa$,
the lower prevision $\LC:\lset\rightarrow\rset$ defined as
\begin{eqnarray}
\label{eq:Choquet_intext}
\LC (X)=\cint X\dlp,\ \forall X\in\lset
\end{eqnarray}
is an extension of $\lpr$ that is $2$-coherent on $\lset$.
\end{proposition}
\begin{proof}
$\LC$ extends $\lpr$:
in fact $\asetpa\subset\lset$ (identifying every event $A$ with its indicator $I_A$),
and $\LC(I_A)=\lpr(A)$ by Proposition \ref{pro:Choquet_int_prop} a).

To prove that $\LC$ is $2$-coherent on $\lset$,
note that $\lset$ is a special set $\lcond$
(Definition \ref{def:def_dlin}),
with $\xset=\lset$, $\bset=\{\Omega,\varnothing\}$.
Therefore,
we may apply Theorem \ref{thm:characterisation} b) and
check that $\LC$ verifies the following axioms:
\begin{itemize}
\item[(A1')]
$\LC(X)-\LC(Y)\leq\sup(X-Y)$, $\forall X,Y\in\lset$
\item[(A2')]
$\LC(\lambda X)=\lambda\LC(X)$, $\forall X\in\lset$, $\forall\lambda\geq 0$.
\item[(A4')]
$\LC(X-\LC(X))=0$, $\forall X\in\lset$.
\item[(A5')]
$\LC(\lambda X)\leq\lambda\LC(X)$, $\forall X\in\lset$, $\forall\lambda<0$.
\end{itemize}
As for condition (A1'), it is equivalent to
\begin{itemize}
\item[(A1'')]
If $X, Y\in\lset$, $k\in\rset$ are such that $Y\geq X+k$,
then $\LC(Y)\geq\LC(X)+k$,
\end{itemize}

In fact,
suppose (A1') holds and take $X,Y\in\lset$, $k\in\rset$ such that
$Y\geq X+k$.
Then, by (A1'), $\LC(X)-\LC(Y)\leq\sup(X-Y)\leq -k$,
hence (A1'').

Conversely, $\forall X,Y\in\lset$,
it holds that $Y\geq X+\inf(Y-X)$.
If (A1'') holds,
this implies $\LC(Y)\geq\LC(X)+\inf(Y-X)$,
equivalent to $\LC(X)-\LC(Y)\leq\sup(X-Y)$,
which is (A1').

Then (A1'') holds because of Proposition \ref{pro:Choquet_int_prop} d), e),
and \eqref{eq:Choquet_intext}.
%\begin{eqnarray*}
%\LC(X)=\cint X\dlp\geq\cint (Y+k)\dlp=\cint Y\dlp + k=\LC(Y)+k.
%\end{eqnarray*}

For (A2'), apply Proposition \ref{pro:Choquet_int_prop} b), while
%\begin{eqnarray*}
%\LC(\lambda X)=\cint \lambda X\dlp=\lambda\cdot\cint X\dlp=\lambda\cdot\LC(X).
%\end{eqnarray*}
axiom (A4') follows from Proposition \ref{pro:Choquet_int_prop} d).
%\begin{eqnarray*}
%\LC(X-\LC(X))=\cint (X-\LC(X))\dlp=\cint X\dlp-\LC(X)=0.
%\end{eqnarray*}

To prove (A5'),
recall that defining $\UP(A)=1-\LP(A)$,
$\forall A\in\asetpa$,
it is necessary for the $2$-coherence of $\lpr$ that $\UP(A)\geq\LP(A)$
\cite[Appendix B, Theorem B3 (b)]{wal91}.
Hence,
using this and the monotonicity of the Riemann integral at the inequality,
either \eqref{eq:Choquet_intext} or Proposition \ref{pro:Choquet_int_prop} f) at the equalities,
\begin{eqnarray*}
\begin{array}{ll}
\displaystyle\LC(X)&\displaystyle=\cint X\dlp=\inf X+\rint{\inf X}{\sup X}\lpr(X\geq x)dx\\
	&\displaystyle\leq\inf X+\rint{\inf X}{\sup X}\UP(X\geq x)dx=\cint X\dup.
\end{array}
\end{eqnarray*}
Finally, take $\lambda<0$.
Applying \eqref{eq:Choquet_intext} at the first equality below,
Proposition~\ref{pro:Choquet_int_prop} b) at the third,
Proposition~\ref{pro:Choquet_int_prop}~c) at the fourth and
the just established
$\cint X\dlp\leq\cint X\dup$ at the inequality,
we obtain
\begin{eqnarray*}	
\begin{array}{ll}
		 \displaystyle\LC(\lambda X)&\displaystyle=\cint \lambda X\dlp=\cint (-\lambda)(-X)\dlp=(-\lambda)\cdot\cint (-X)\dlp=\\
		&\displaystyle=\lambda\cdot\cint X\dup\leq\lambda\cdot\cint X\dlp=\lambda\cdot\LC(X),
	\end{array}
\end{eqnarray*}
which is (A5').
\end{proof}

\begin{remark}[$2$-convexity of the Choquet integral]
\label{rem:2_convex_CI}
As a byproduct of the equivalence between (A1'') and (A1') shown in the proof of Proposition \ref{pro:Choquet_2c}, we get that the Choquet integral $\LC$ in \eqref{eq:Choquet_intext} is $2$-convex on $\lset$,
as soon as $\lpr$ in \eqref{eq:Choquet_intext} is \emph{any} normalised capacity defined on $\asetpa$. To prove this relaxation of Proposition \ref{pro:Choquet_2c}, note that
(A1') and (A4'), which hold because of Proposition~\ref{pro:Choquet_int_prop} d), e) and the equivalence mentioned above, characterise $2$-convexity of $\LC$, by Theorem \ref{thm:characterisation} d).
This implies also, by Propositioin \ref{pro:Choquet_int_prop} a), that \emph{any} normalised capacity is always $2$-convex on $\asetpa$.
\end{remark}

The Choquet integral extension $\LC$ is therefore a $2$-coherent extension which
is relatively simple to compute,
by Proposition \ref{pro:Choquet_int_prop} f) or g).
In the case of a simple $X$, in particular, fewer computations are needed to determine $\LC$ rather than $\LEDC$.
The two extensions may not coincide, as shown next:
\begin{example}
	\label{exa:not_coincide}
	Let $\prt$, $\lpr$ and $Z$ be as in Example \ref{exa:comparison_2c_2c}.
	By Proposition \ref{pro:Choquet_int_prop} g),
	\begin{eqnarray*}
		\begin{array}{ll}
			\LC(X)&=-1+(0-(-1))\lpr(b\vee c\vee d)+1\lpr(c\vee d)+2\lpr(d)\\
			&=0.4>\LED(X)=0.
		\end{array}
	\end{eqnarray*}
\end{example}

\textbf{Discussion.}
Hence,
$\LC$ is intrinsically `more precise' than $\LED$.
More generally,
the following holds:
\begin{eqnarray}
	\label{eq:Choquet_ext_order}
	\LE\geq\LC\geq\LED\geq\LEDC.
\end{eqnarray} 
Note that each inequality may be strict:
$\LED>\LEDC$ in Example \ref{exa:different_ne},
$\LC>\LED$ in Example \ref{exa:not_coincide},
while it is known that $\LE=\LC$, $\forall X\in\lset$,
iff $\lpr$ is coherent and $2$-monotone
\cite[Theorem 6.1]{wal81} (otherwise $\LE>\LC$ for at least one $X\in\lset$).
Therefore, Sandwich Theorems analogous to those of Sections
\ref{sec:extensions_full}, \ref{sec:extension_cond_low_prev} 
seem not to be available in the extension problems studied in this section.

Thus, if $\lpr$ is coherent and $2$-monotone
\eqref{eq:Choquet_ext_order} simplifies a little, since then $\LE=\LC$.
Yet, even the rather strong assumption of $2$-monotonicity does not ensure that the Choquet integral extension may be equal to any of the remaining extensions we are investigating.
This is shown in the next example.
\begin{example}
\label{ex:2_monotone_Choquet}
Given the partition $\prt=\{\omega_1,\omega_2,\omega_3\}$,
define the lower probability $\lpr$ on $\asetpa$ as the lower envelope of the precise probabilities $P_1$, $P_2$:
$\lpr(A)=\min\{P_1(A),P_2(A)\}$, $\forall A \in\asetpa$,
see (the shaded area of) Table \ref{tab:2_monotone_Choquet}.
Let also $X$ be the gamble in $\lset$ defined by $X(\omega_1)=2$, $X(\omega_2)=3$, $X(\omega_3)=0$.

\begin{table}[ht]
\centering 
\begin{tabular}{c||c|c|c|c|c|c|c|c|c|c|}
	%\hline 
	$\,$ &$\omega_1$  &$\omega_2$  &$\omega_3$  &$\omega_1\vee\omega_2$  &$\omega_1\vee\omega_3$  &$\omega_2\vee\omega_3$  &$X$  &$\alpha_{P_{i}}$  \\ 
	\hline 
	\cellcolor{Gray}$P_1$ &\cellcolor{Gray}$0.3$  &\cellcolor{Gray}$0.4$  &\cellcolor{Gray}$0.3$  &\cellcolor{Gray}$0.7$  &\cellcolor{Gray}$0.6$  &\cellcolor{Gray}$0.7$  &$1.8$  &$0$  \\ 
	\hline 
	\cellcolor{Gray}$P_2$ &\cellcolor{Gray}$0.1$  &\cellcolor{Gray}$0.7$  &\cellcolor{Gray}$0.2$  &\cellcolor{Gray}$0.8$  &\cellcolor{Gray}$0.3$  &\cellcolor{Gray}$0.9$ &$2.3$  &$0$  \\ 
	\hline 
	$P_3$&$0$  &$0$  &$1$  &$0$  &$1$  &$1$  &$0$  &$0.7$  \\ 
	\hline 
	\cellcolor{Gray}$\lpr$ &\cellcolor{Gray}$0.1$  &\cellcolor{Gray}$0.4$  &\cellcolor{Gray}$0.2$  &\cellcolor{Gray}$0.7$  &\cellcolor{Gray}$0.3$  &\cellcolor{Gray}$0.7$  &$\,$  &$\,$  \\ 
	\hline 
	$\lpr^{\prime}$ &$0.1$  &$0.4$  &$0.2$  &$0.7$  &$0.3$  &$0.7$  &$0.7$  &$\,$ 
%	\hline 
\end{tabular} 
\caption{Data for Example \ref{ex:2_monotone_Choquet} }
\label{tab:2_monotone_Choquet}
\end{table}
\begin{itemize}
	\item[$\bullet$]
	$\lpr$ is coherent on $\asetpa$,
	as a lower envelope of (precise) probabilities \cite[Corollary 3.3.4]{wal91},
	and $2$-monotone,
	being coherent and defined on a three-element partition
	(see e.g. \cite[Proposition 6.9]{tro14}).
	\item[$\bullet$]
	Applying Proposition \ref{pro:Choquet_int_prop} g),
	$\LC(X)=0+(2-0)\lpr(X\geq 2)+(3-2)\lpr(X\geq 3)=1.8$.
	As for $\LED(X)$, we obtain $\LED(X)=1.4<\LC(X)$,
	using e.g. \eqref{eq:2cohneexpr1}
	(the value $1.4$ is achieved with $E=\omega_1\vee\omega_2$:
	$\inf(X|\omega_1\vee\omega_2)\lpr(\omega_1\vee\omega_2)+\inf(X|\omega_3)(1-\lpr(\omega_1\vee\omega_2))=2\cdot 0.7+0=1.4$).
	\item[$\bullet$]
	To compute $\LEC(X)$, consider the extension of $\lpr$ on $\eset=\asetpa\cup\{X\}$ given by $\lpr^{\prime}(\cdot)=\min\{P_1(\cdot),P_2(\cdot),P_3(\cdot)+0.7\}$ 
	(see Table \ref{tab:2_monotone_Choquet}).
	Note that when the argument of $P_1$, $P_2$, $P_3$ is an event in  $\asetpa$,
	$\lpr^\prime=\lpr$,
	while $\lpr^{\prime}(X)=0.7$.
	Thus, $\lpr^\prime$ is not coherent on $\eset$,
	since $\lpr^{\prime}(X)<\LE(X)=1.4$.
	By contrast, $\lpr^\prime$ is centered convex on $\eset$:
	this follows from the Envelope Theorem for convex previsions \cite[Theorem 3.3]{pel03}.
	This theorem states that $\lpr^\prime$ is a centered convex lower prevision on a set $\dset$ of gambles if and only if there exists a non-empty set of (precise) previsions
	$\mathcal{P}$ such that,
	$\forall X\in\dset$, $\lpr^{\prime}(X)=\min_{P\in\mathcal{P}}\{P(X)+\alpha_{P}\}$,
    with
	$\min_{P\in\mathcal{P}}\{\alpha_P\}=0$. 
	In this example $\mathcal{P}=\{P_1, P_2, P_3\}$, $\alpha_{P_{1}}=\alpha_{P_{2}}=0$, $\alpha_{P_{3}}=0.7$.
	Consequently,
	$\LEC(X)\leq\lpr^{\prime}(X)=0.7$
	(by the least-committal property of the convex natural extension,
	cf. the corresponding property stated for $2$-convexity in Proposition \ref{pro:properties_2_ne} e)).
	On the other hand, note that $X=2\omega_1+3\omega_2\geq\omega_1+\omega_2$.
	Monotonicity of convex previsions implies then $\LEC(2\omega_1+3\omega_2)\geq\LEC(\omega_1+\omega_2)=\lpr(\omega_1\vee\omega_2)=0.7$.
	Hence $\LEC(X)=0.7$.
	\item[$\bullet$]
	Since $\LEC(X)\geq\LEDC(X)$ by Lemma \ref{lem:order_ne},
	and $\LEDC(X)\geq 0.7$ (again by the monotonicity argument),
	also $\LEDC(X)=0.7$.
\end{itemize}
Summoning up, in this example $\lpr$ is $2$-monotone on $\asetpa$ but
$\LE(X)=\LC(X)>\LED(X)>\LEC(X)=\LEDC(X)$.
\end{example}
Interestingly,
\eqref{eq:Choquet_ext_order} and the preceding results are helpful in better delimiting the role of the Choquet integral extension $\LC$ in Imprecise Probability Theory.
In fact,
as well known $\LC$ is a lower bound to the natural extension $\LE$,
hence generally not a coherent extension.
With the $2$-coherent natural extension, the role of $\LC$ is reversed:
it is an upper bound to $\LEDC$,
and is anyway always $2$-coherent (if the starting $\lpr$ is).
We might prefer $\LC$ to $\LED$ for computational reasons,
in order to have a `more precise' extension,
or if we require comonotone additivity.
Concerning the last motivation,
recall that a measure $\mu$ is \emph{comonotone additive} if it is additive on comonotone gambles.
That is, $\mu(X+Y)=\mu(X)+\mu(Y)$ if $X$, $Y$ are such that there are no
$\omega_1$, $\omega_2$ with $X(\omega_1)<X(\omega_2)$ and $Y(\omega_1)>Y(\omega_2)$.
Then,
it is well known that the Choquet integral, hence $\LC$, is comonotone additive
\cite[Proposition C.5 (vii)]{tro14}.
Finally,
note that requiring comonotone additivity with $2$-coherence is much less demanding than with coherence.
In fact,
comonotone additivity is equivalent to $2$-monotonicity under coherence 
\cite[Theorem 12]{dec05},\cite[Theorem 6.22]{tro14}.

\begin{remark}
\label{rem:relaxing}
Suppose we relax the assumptions from the beginning of this section
by requiring only that $\lpr$ is centered $2$-convex on $\asetpa$.
As an interesting follow-up of Remark \ref{rem:2_convex_CI},
$\LC$ is then a convex extension of $\lpr$,
and $\LC\geq\LEDC$.
However,
when $\LED$ is finite,
cf. the later Proposition \ref{finiteness_AUL},
$\LC\geq\LED\geq\LEDC$
(both inequalities are strict in Example \ref{ex:2_monotone_Choquet}).
Then, $\LC$ approximates $\LEDC$ more loosely than $\LED$.
\end{remark}

\section{More on the ordering of natural extensions}
\label{sec:more_order_ne}
In this section we deepen some aspects,
which were mostly out of the scope of the previous three sections, of the various natural extensions
we are dealing with.
The setting is an unconditional environment for the sake of simplicity,
thus $\lpr(\cdot)$ is now a lower prevision on an arbitrary set $\dset$ of (unconditional) gambles.
We focus on
\begin{itemize}
	\item[a)]
	relationships with avoiding sure loss-type notions;
	\item[b)]
	conditions for the finiteness of the natural extensions.
\end{itemize}
Items a) and b) are strongly connected:
as we shall now see,
precisely conditions of avoiding sure loss (in some sense) are relevant for b).
\begin{definition}
	\label{def:ASL_AUSL}
	Given $\lpr:\dset\rightarrow\rset$,
	\begin{itemize}
		\item[a)] $\lpr$ \emph{avoids sure loss} (ASL) on $\dset$ \cite[Section 2.4.4 (a)]{wal91} iff $\forall n\in\nset^+$, $\forall X_1,\ldots,X_n\in\dset$, $\forall s_1,\ldots,s_n\geq 0$,
		defining $\LG=\sum_{i=1}^{n}s_i(X_i-\lpr(X_i))$,
		it holds that $\sup\LG\geq 0$.
		\item[b)]
		$\lpr$ \emph{$1$-avoids sure loss} (1-ASL) on $\dset$ \cite[Remark 2]{pel16} iff a) applies, with $n=1$.
		\item[c)]
		$\lpr$ \emph{avoids unbounded sure loss} (AUSL) on $\dset$ \cite[Definition 5]{pel03bis} iff there exists $k\in\rset$ such that,
		$\forall n\in\nset^+$, $\forall X_1,\ldots,X_n\in\dset$, $\forall s_1,\ldots,s_n\geq 0$ with $\sum_{i=1}^{n}s_i=1$,
		defining $\LG=\sum_{i=1}^{n}s_i(X_i-\lpr(X_i))$, it holds that $\sup\LG\geq k$.
		\item[d)]
		$\lpr$ \emph{$1$-avoids unbounded sure loss} ($1$-AUSL) on $\dset$ iff
		c) applies with $n=1$ (and, consequently, $s_1=1$).
	\end{itemize}
\end{definition}
The condition of avoiding sure loss is the best known and strongest of the four.
$1$-ASL is in fact a rather weak consistency requirement. It is equivalent to $\lpr(X)\leq\sup X$, $\forall X\in\dset$,
as can be easily checked.
It is also easy to see that $\lpr$ is $1$-AUSL iff there exists $k\in\rset$ such that $\underline{Q}=\lpr+k$ is $1$-ASL.
As for AUSL, it holds that
\begin{proposition}
	\label{pro:suff_AUSL}
	Each of the following conditions is sufficient for $\lpr:\dset\rightarrow\rset$ to avoid unbounded sure loss:
	\begin{itemize}
		\item[a)]
		$\dset$ is finite
		\item[b)]
		$\lpr$ is convex on $\dset$
	\end{itemize}
\end{proposition}
\begin{proof}
	\begin{itemize}
		\item[a)]
		If $\dset=\{X_1,\ldots,X_n\}$,
		any $\LG$ in Definition \ref{def:ASL_AUSL} c) may be written in the form
		$\LG=\sum_{i=1}^{n}s_i(X_i-\lpr(X_i))$,
		with some $s_i$ possibly equal to zero.
		Since anyway $\sum_{i=1}^{n}s_i=1$,
		$\LG$ is a convex linear combination of $X_i-\lpr(X_i)$, $i=1,\ldots,n$.
		As such, $\LG\geq\min_{i=1,\ldots,n}\{\inf X_i-\lpr(X_i)\} =k$.
		Thus, $\sup\LG\geq k$.
		\item[b)]
		Proven in \cite[Proposition 4]{pel03bis}.
	\end{itemize}
\end{proof}
Indeed, AUSL may be a mild consistency requirement.
As appears from Proposition \ref{pro:suff_AUSL} a), it imposes no constraint whatever on $\lpr$ when its domain is finite.
Even with infinite domains, it may be expected to exhibit some unsatisfactory features.
Consider for this that also the stronger ASL condition is compatible with lack of monotonicity,
unlike $2$-coherence and even $2$-convexity.

In our view,
the conditions of ASL, $1$-ASL, AUSL and $1$-AUSL have an ancillary role with respect to coherence, $2$-coherence, convexity and $2$-convexity, respectively:
they ensure finiteness of the corresponding natural extensions.
In fact,
\begin{proposition}
	\label{finiteness_AUL}
	Given $\lpr:\dset\rightarrow\rset$,
	\begin{itemize}
		\item[a)]
		$\LE(Z)$ is finite $\forall Z$ iff $\lpr$ avoids sure loss.
		\item[b)]
		$\LED(Z)$ is finite $\forall Z$ iff $\lpr$ $1$-avoids sure loss.
		\item[c)]
		$\LEC(Z)$ is finite $\forall Z$ iff $\lpr$ avoids unbounded sure loss.
		\item[d)]
		$\LEDC(Z)$ is finite $\forall Z$ iff $\lpr$ $1$-avoids unbounded sure loss.
	\end{itemize}
\end{proposition}
\begin{proof}
Item a) was proven in \cite[Section 3.1.2 (a)]{wal91},
b) in \cite[Proposition 12 and Remark~3]{pel16},
c) in \cite[Proposition 3]{pel03bis}.

We prove here d).
Recall for this (Definition \ref{def_nat_ext} d)) that
\begin{eqnarray*}
		\begin{array}{ll}
			\LEDC(Z)&=\sup\LDC(Z)\\
			&=\sup\{\alpha\in\rset:\sup\{(X-\lpr(X))-(Z-\alpha)\}<0, \mbox{for some } X\in\dset\}.
		\end{array}
\end{eqnarray*}
Preliminarily,
we prove that $\LDC(Z)\neq\emptyset$, $\forall Z$.
In fact, take $X\in\dset$.
For $\alpha<\inf Z-\sup X+\lpr(X)$ it holds that
$\sup\{X-\lpr(X)-(Z-\alpha)\}\leq\sup X-\lpr(X)-\inf Z+\alpha<-\alpha+\alpha=0$,
hence $\alpha\in\LDC(Z)$.

We prove now the direct implication:
assuming that $\LEDC$ is finite,
suppose by contradiction that $\lpr$ is not $1$-AUSL.
Then,
$\forall k\in\rset$, $\exists X\in\dset:\sup(X-\lpr(X))<k$.
It follows that,
for any $\varepsilon>0$,
$(X-\lpr(X))-(Z-(\inf Z-k-\varepsilon))\leq\sup\{X-\lpr(X)\}-(Z-\inf Z)-k-\varepsilon<-\varepsilon<0$.

Therefore,
$\sup\{(X-\lpr(X))-(Z-(\inf Z-k-\varepsilon))\}\leq-\varepsilon<0$,
meaning that $\inf Z-k-\varepsilon\in\LDC(Z)$, $\forall k\in\rset$, $\forall\varepsilon>0$.
Hence $\LEDC(Z)=+\infty$.

Conversely,
let $\lpr$ be $1$-AUSL and consider any $\alpha\in\LDC(Z)$.
Then,
by definition of $\LEDC(Z)$,
\begin{eqnarray*}
\sup\{(X-\lpr(X))-(Z-\alpha)\}<0.
\end{eqnarray*}
Hence $X-\lpr(X)\leq\sup Z-\alpha$.
Taking the supremum and since $\lpr$ is $1$-AUSL,
there exists $k$ such that
\begin{eqnarray*}
k\leq\sup\{X-\lpr(X)\}\leq\sup Z-\alpha.
\end{eqnarray*}
Since then $\alpha\leq\sup Z-k$,
$\forall\alpha\in\LDC(Z)$,
we have that $\LEDC(Z)\in\rset$.
\end{proof}
Note that each of $\LE$, $\LED$, $\LEC$, $\LEDC$ is finite (or alternatively infinite) for all gambles if it is so for one gamble $Z$.
In other words, these extensions cannot be simultaneously finite on some gambles and infinite on other ones.

Proposition \ref{finiteness_AUL} clarifies what conditions characterise the finiteness of $\LE$, $\LED$, $\LEC$, $\LEDC$.
From Lemma \ref{lem:order_ne}, we deduce that avoiding sure loss is sufficient for the finiteness of them all.

Further,
Proposition \ref{finiteness_AUL} is helpful in establishing that some natural extensions may be infinite while other ones are not for the same gamble,
%at the same time,
and that there is no order relation, in general, between $\LED$ and $\LEC$
(thus the ordering among $\LE$, $\LE_2$, $\LEC$, $\LEDC$ is partial,
following equations \eqref{eq:order_ne}).
We include some examples proving these facts:
\begin{itemize}
	\item[$-$]
	$\LED=+\infty>\LEC$ in Example \ref{ex:ne_order_1} a);
	\item[$-$]
	$\LE=+\infty>\LEC>\LED$ in Example \ref{ex:ne_order_1} b)
	(while $+\infty>\LED>\LEC$ in Example \ref{ex:2_monotone_Choquet});
	\item[$-$]
	$\LEC=+\infty>\LED$ in Example \ref{ex:ne_order_2} a);
	\item[$-$]
	$\LEC=\LED=+\infty>\LEDC$ in Example \ref{ex:ne_order_2} b).
\end{itemize}
\begin{example}
	\label{ex:ne_order_1}
	Let $\prt=\{\omega_1,\ldots,\omega_n\}$ be a finite partition,
	and $\lpr:\prt\rightarrow\rset$ be a lower probability,
	with $\lpr(\omega_i)=\frac{1}{n}+\delta$, $\delta>0$, $i=1,\ldots,n$.
	\begin{itemize}
		\item[a)]
		Suppose additionally that $\delta>1-\frac{1}{n}$. Then,
		\begin{itemize}
			\item[$\bullet$]
			$\lpr$ is not $1$-ASL:
			$\lpr(\omega_i)>\frac{1}{n}+1-\frac{1}{n}=1=\sup\omega_i$,
			$i=1,\dots,n$.
			\item[$\bullet$]
			$\lpr$ is AUSL, by Proposition \ref{pro:suff_AUSL} a).
		\end{itemize}
		Therefore, by Proposition \ref{finiteness_AUL},
		$\LED=+\infty>\LEC$.
		\item[b)]
		Now suppose instead $0<\delta<\frac{1}{n(n-1)}$, $n\geq 3$.
		Then
		\begin{itemize}
			\item[$\bullet$]
			$\lpr$ is not ASL:
			$\sum_{i=1}^{n}\lpr(\omega_i)>1$,
			contradicting a necessary condition for ASL
			\cite[Section 4.6.1]{wal91}.
			\item[$\bullet$]
			$\lpr$ is AUSL, by Proposition \ref{pro:suff_AUSL} a).
			Hence, $\LEC(Z)<+\infty$, $\forall Z$.
			In particular,
			$\LEC(\varnothing)\geq\delta>0$.
			To see this, use Definition \ref{def_nat_ext} c) with
			$s_i=\frac{1}{n}$, $i=1,\ldots,n$:
			$\sup\{\sum_{i=1}^{n}\frac{1}{n}(\omega_i-\lpr(\omega_i))-(0-\alpha)\}<0$
			iff $\alpha<\inf\{\frac{1}{n}\sum_{i=1}^{n}\lpr(\omega_i)-\frac{1}{n}\sum_{i=1}^{n}\omega_i\}=
			%\inf\{\frac{1}{n}n(\frac{1}{n}+\delta)-\frac{1}{n}\Omega\}=
			\delta$.
			Hence any $\alpha<\delta$ belongs to $L_c(\varnothing)$ and $\LEC(\varnothing)\geq\delta>0$.
			\item[$\bullet$]
			$\lpr$ is $2$-coherent on $\prt$:
			by Definition \ref{def:consistency} b),
			any admissible $\LG$ involves two atoms of $\prt$,
			let them be $\omega_i$, $\omega_j$
			(not necessarily distinct: we neglect the trivial case $\omega_i=\omega_j$).
			Since $n\geq 3$ and $\lpr(\omega_i)+\lpr(\omega_j)<1$
			(because $\delta<\frac{1}{n(n-1)})$,
			for any choice of $\omega_i, \omega_j\in\prt$ ($i\neq j)$ $\lpr$ is the restriction on
			$\{\omega_i,\omega_j\}$ of a dF-coherent (precise) probability
			$P$ concentrated on $\omega_i$, $\omega_j$, $\omega_k\ (k\neq i, k\neq j$):
			$P(\omega_i)=\lpr(\omega_i)$, $P(\omega_j)=\lpr(\omega_j)$, $P(\omega_k)=1-(P(\omega_i)+P(\omega_j))$.
			Since dF-coherence implies coherence and $2$-coherence,
			any $\LG$ admissible for $2$-coherence satisfies the condition $\sup\LG\geq 0$, $\forall\omega_i,\omega_j$.
			Therefore $\lpr$ is $2$-coherent and
			$\LED$ is finite and $2$-coherent too,
			which implies $\LED(\varnothing)=0$.	
		\end{itemize}
	From what derived above and Proposition \ref{finiteness_AUL},
	we conclude that $\LE(\varnothing)=+\infty>\LEC(\varnothing)>\LED(\varnothing)$.
	\end{itemize}
	
\end{example}
\begin{example}
	\label{ex:ne_order_2}
	Given $\prt=\{\omega_1,\omega_2\}$,
	define $\dset=\{X_0,X_1,\ldots,X_{2n},X_{2n+1},\ldots,\}$ as follows,
	for any $n\in\nset$:
	\begin{itemize}
		\item[$-$] $X_{2n}(\omega_1)=-n$, $X_{2n}(\omega_2)=-2n$;
		\item[$-$] $X_{2n+1}(\omega_1)=-2n$, $X_{2n+1}(\omega_2)=-n$. 
	\end{itemize}
	Then assign a lower prevision $\lpr:\dset\rightarrow\rset:
	\lpr(X_{2n})=\lpr(X_{2n+1})=-n+\delta$, $\delta\geq 0$, for any $n\in\nset$.
	We shall consider the gains
	\begin{eqnarray*}
		\LG_n=\frac{1}{2}(X_{2n}-\lpr(X_{2n}))+\frac{1}{2}(X_{2n+1}-\lpr(X_{2n+1})).
	\end{eqnarray*}
	Note that, for any $n\in\nset$,
	\begin{eqnarray}
	\label{eq:gain_example}
	\LG_n(\omega_1)=\LG_n(\omega_2)=-\frac{n}{2}-\delta,
	\end{eqnarray}
	as is easily verified.
	
	Then,	
	%\begin{itemize}
	%\item[$\bullet$]
	\emph{$\lpr$ is not AUSL}:
	consider the sequence $(\LG_n)_{n\in\nset}$ of gains admissible by Definition \ref{def:ASL_AUSL} c).
	Since $\lim_{n\rightarrow +\infty}\sup G_n=\lim_{n\rightarrow +\infty}(-\frac{n}{2}-\delta)=-\infty$,
	there is no $k\in\rset$ such that $\sup\LG_n\geq k,\ \forall n\in\nset$.
	
	Let us now investigate two distinct subcases.
	\begin{itemize}
		\item[a)]
		$\delta=0$.
		
		In this case $\lpr$ is $1$-ASL:
		$\lpr(X_{2n})=\lpr(X_{2n+1})=-n=\sup(X_{2n})=\sup(X_{2n+1})$.
		
		From Proposition \ref{finiteness_AUL},
		we obtain $\LEC=+\infty>\LED$.
		\item[b)]
		$\delta>0$
		
		Then $\lpr$ is not $1$-ASL:
		$\lpr(X_{2n})=\lpr(X_{2n+1})=-n+\delta>\sup(X_{2n})=\sup(X_{2n+1})=-n$.
		
		From Proposition \ref{finiteness_AUL},
		$\LEC=\LED=+\infty$.
		However,
		note that $\LEDC(\varnothing)<+\infty$.
		In fact,
		from Definition \ref{def_nat_ext} d),
		\begin{eqnarray*}
			\LDC(\varnothing)=\{\alpha\in\rset:\sup\{X_n-\lpr(X_n)+\alpha<0\},\mbox{ for some } X_n\in\dset\}.
		\end{eqnarray*}
		Since $\sup\{X_n-\lpr(X_n)+\alpha\}=\alpha+\sup\{X_n-\lpr(X_n)\}=\alpha-\delta<0$
		iff $\alpha<\delta$,
		we obtain $\LEDC(\varnothing)=\delta\in\rset^+$.
		
		Summarising,
		we have that
		$\LE(\varnothing)=\LEC(\varnothing)=\LED(\varnothing)=+\infty>\LEDC(\varnothing)$.
	\end{itemize}
	
	%\end{itemize}
	
\end{example}

\section{Conclusions}
\label{sec:conclusions}
In this paper we have deepened the study of the $2$-coherent and $2$-convex natural extensions $\LED$ and $\LEDC$,
initiated in \cite{pel16, pel17},
and have established the role of the Choquet integral extension within $2$-coherence.
We have seen in Sections \ref{sec:extensions_full} and \ref{sec:extension_cond_low_prev}
that there are important situations where $\LEDC$ or $\LED$ coincide with the natural extension $\LE$ or the convex natural extension $\LEC$.
This does not happen, generally, in the framework of Section \ref{sec:extension_lprob_to_lprev},
i.e. if $\lpr$ is given on $\asetpa$
and the extension is on $\lset$.
Intuitively,
this might be justified thinking that while the environment where $\lpr$ is initially assessed is `full' in some sense in all cases we discuss,
in the hypotheses of Section \ref{sec:extension_lprob_to_lprev}
it is relatively `poor' with respect to the extended set.
In fact there is a jump there from events to gambles,
and from probabilities to previsions.
We supplied in this case simplifying formulae for $\LED$ and $\LEDC$,
which make the practical computation of these extensions rather simple
when the partition $\prt$ is finite,
and more generally show that the extensions depend on the evaluation of decumulative distribution functions.
As for future work,
the most immediate challenge is to investigate whether the properties detected so far extend to further situations.
More work is needed also to practically exploit the cases of coincidence among different natural extensions.
Lastly,
other unexplored areas concern the `transitivity property' (cf. \cite[Section 4.5.4 ]{tro14}) for the different natural extensions and the potential role (for instance in computing extensions directly) of significant envelope theorems for $2$-coherent or $2$-convex previsions.
These theorems should generalise the existing envelope theorems for, respectively, coherent lower previsions (cf. \cite[Section 3.3.3]{wal91}) and convex lower previsions (cf. the statement of the envelope theorem recalled in Example~\ref{ex:2_monotone_Choquet}).

\section*{Acknowledgements}
We wish to thank the anonymous referees for their helpful suggestions,
which let us improve the paper.

We acknowledge partial support by the FRA2015 grant `Mathematical Models for Handling Risk and Uncertainty'.

\bibliographystyle{plain}

\end{document}